\documentclass[reqno]{amsart}

\usepackage{amsmath,amssymb,amsthm,comment,enumerate,mathrsfs}

\usepackage{hyperref}

\usepackage{color}

\numberwithin{equation}{section}

\newtheorem{Thm}{Theorem}[section]
\newtheorem{Prop}[Thm]{Proposition}
\newtheorem{Lem}[Thm]{Lemma}
\newtheorem{Cor}[Thm]{Corollary}

\theoremstyle{definition}
\newtheorem{Rem}[Thm]{Remark}
\newtheorem{Def}[Thm]{Definition}
\newtheorem{Expl}[Thm]{Example}

\newcommand{\R}{\mathbb{R}}

\newcommand{\E}[1]{{\rm E}(#1)}

\newcommand{\cA}{\mathscr{A}}

\newcommand{\es}{\emptyset}

\newcommand{\Lip}{\operatorname{Lip}}

\newcommand{\ol}{\overline}
\renewcommand{\rho}{\varrho}

\newcommand{\sig}{\sigma}

\title{
Weakly Externally Hyperconvex Subsets
and Hyperconvex Gluings
}

\author{
Benjamin Miesch and
Ma\"el Pav\' on}
\address{Department of Mathematics, ETH Z\"urich, 8092 Z\"urich, Switzerland}
\email{benjamin.miesch@math.ethz.ch}
\email{mael.pavon@math.ethz.ch}

\date{\today}

\begin{document}


\begin{abstract}
	We give a necessary and sufficient condition for gluings of hyperconvex metric spaces along weakly externally hyperconvex subsets in order that the resulting space be hyperconvex. This leads to a full characterization of gluings of two isometric copies of the same hyperconvex space.
	
	Furthermore, we investigate the case of gluings of finite dimensional hyperconvex linear spaces along linear subspaces. For this purpose, we characterize the convex polyhedra in $l_\infty^n$ which are weakly externally hyperconvex.
\end{abstract}

\maketitle


\section{Introduction}

Hyperconvexity can be regarded as a general metric notion of weak non-positive curvature. In relation with Isbell's injective hull, developing tools for hyperconvex metric spaces can lead to improvements and extensions of analogue results of CAT(0) geometry.
The present work provides a new source of hyperconvex metric spaces via gluing along weakly externally hyperconvex subsets and proves new results on weakly externally hyperconvex subsets of $l_{\infty}(I)$.

A metric space $(X,d)$ is called \emph{hyperconvex} if for any collection $\{ B(x_i,r_i)\}_{i\in I}$ of closed balls with $d(x_i,x_j) \leq r_i + r_j$, for all $i,j \in I$, we have $\bigcap_{i\in I} B(x_i,r_i) \neq \emptyset$.	Hyperconvex spaces are the same as absolute 1-lipschitz retracts or injective metric spaces, see \cite{Aro}.

A non-empty subset $A$ of a metric space $(X,d)$ is \emph{externally hyperconvex} in $X$ if for any collection of closed balls $\{B(x_i,r_i)\}_{i \in I}$ in $X$ with $d(x_i,x_j) \leq r_i+r_j$ and $d(x_i,A)\leq r_i$, for all $i,j \in I$, we have $A\cap \bigcap_{i\in I} B(x_i,r_i) \neq \emptyset$ and $A$ is \emph{weakly externally hyperconvex} if for any $x \in X$, $A$ is externally hyperconvex in $A \cup \{x\}$.

We first extend the work initiated in \cite{Mie} by considering gluings along weakly externally hyperconvex subsets:

\begin{Thm}\label{thm:gluing along WEH}
	Let $(X,d)$ be the metric space obtained by gluing a family of hyperconvex metric spaces $(X_\lambda,d_\lambda)_{\lambda \in \Lambda}$ along some set $A$ such that for each $\lambda\in \Lambda$, $A$ is weakly externally hyperconvex in $X_\lambda$. Then, $X$ is hyperconvex if and only if for all $\lambda \in \Lambda$ and all $x \in X \setminus X_\lambda$, the set $B(x,d(x,A))\cap A$ is externally hyperconvex in $X_\lambda$. Moreover, if $X$ is hyperconvex, the subspaces $X_\lambda$ are weakly externally hyperconvex in $X$.
\end{Thm}

We go on by showing that the above sufficient condition for hyperconvexity of the gluing is necessary in case isometric copies are glued together, namely 

\begin{Thm}\label{thm:gluing of copies}
	Let $X$ be a hyperconvex metric space and let $A$ be a subset. Then $X \sqcup_A X$ is hyperconvex if and only if $A$ is weakly externally hyperconvex in $X$ and for every $x \in X$, the intersection $B(x,d(x,A))\cap A$ is externally hyperconvex in $X$.
\end{Thm}

In particular, Theorem \ref{thm:gluing of copies} shows that the conditions imposed in Theorem \ref{thm:gluing along WEH} are natural. 

In order to be able to apply Theorems~\ref{thm:gluing along WEH} and \ref{thm:gluing of copies} to construct concrete examples of injective metric spaces by gluing, we turn our attention to weakly externally hyperconvex subsets of $l_{\infty}(I)$. A subset $Y$ of a metric space $X$ is called \textit{proximinal} if for any $x \in X$, we have $B(x,d(x,Y)) \cap Y \neq \emptyset$. It is important to note that if $Y$ is weakly exernally hyperconvex in $X$, then it is necessarily proximinal.

\begin{Thm} \label{Thm:CellsInjectiveHullWEH}
Let $I \ne \es$ be any index set. Suppose that $Q$ is a non-empty proximinal subset
of $l_\infty(I)$ given by an arbitrary system of inequalities
of the form $\sig x_i \le C$ or $\sig x_i + \tau x_j \le C$
with $\sig,\tau \in \{ \pm 1 \} $ and $C \in \R$. Then $Q \in \mathcal{W}(l_\infty(I))$.
\end{Thm}

This has consequences on the structure of Isbell's injective hull $\E X$; see Remark~\ref{RemarkInjectiveHull}. 
In the finite dimensional case, we moreover obtain a characterization of weakly externally hyperconvex subspaces. In the following, $l_{\infty}^n$ denotes the vector space $\R^n$ endowed with the supremum norm. Moreover, for any metric space $(X,d)$, a non-empty subset $S \subset X$ is said to be \textit{strongly convex} if for every $x,y \in S$, one has $I(x,y):= \{ z \in X : d(x,z) + d(z,y) = d(x,y)\} \subset S$.

In Theorem~\ref{Thm:FiniteDimensionalWEHLinearSubspaces}, we characterize weakly externally hyperconvex subspaces of $l_{\infty}^n$. We show that they correspond exactly to the linear subspace that can be written as the intersection of hyperplanes having as normal vector a vector of the form $\sigma e_i$ or $ \sigma e_i + \tau e_j$ where $\sigma, \tau \in \{\pm 1\}$ and $i,j \in \{1, \cdots , n \} $. Now, a \textit{convex polyhedron} in a finite dimensional normed linear space is a finite intersection of closed half-spaces. Using Proposition~\ref{Prop:Cuboids} and Theorem~\ref{Thm:FiniteDimensionalWEHConvexPolyhedra}, we obtain
\begin{Thm}
Suppose that $Q$ is a convex polyhedron in $l_{\infty}^n$ with non-empty interior.
Then, the following hold:
\begin{enumerate}[(i)]
\item $Q$ is externally hyperconvex in $l_{\infty}^n$ if and only if $Q$ is a cuboid.
\item $Q$ is weakly externally hyperconvex in $l_{\infty}^n$ if and only if $Q$ is given by a finite system of inequalities of the form $\sig x_i \le C$ or $\sig x_i + \tau x_j \le C$ with $\sig,\tau \in \{ \pm 1 \} $  and $C \in \R$.
\end{enumerate}
\end{Thm}

Finally, we use Theorem~\ref{Thm:FiniteDimensionalWEHLinearSubspaces} to give an explicit characterization for the gluing of $l^n_{\infty}$ and $l^m_{\infty}$ along a linear subspace to be hyperconvex:

\begin{Thm}\label{Thm:CharacterizationGluingLInfinity}
Let $X_1=l_\infty^n$ and $X_2=l_\infty^m$. Moreover let $V$ be a linear subspace of both $X_1$ and $X_2$ such that $V \neq X_1,X_2$. Then
\[
X = X_1 {\sqcup}_V X_2
\]
is hyperconvex if and only if there is some $k$ such that $X_1=l_\infty^k \times X_1'$, $X_2=l_\infty^k \times X_2'$,  $V=l_\infty^k \times V'$ and $V'$ is strongly convex in both $X_1'=l_\infty^{n-k}$ and $X_2'=l_\infty^{m-k}$.
\end{Thm}


\section{Externally and Weakly Externally Hyperconvex Subspaces} 


	In this section we collect a bench of results for externally and weakly externally hyperconvex metric subspaces, their neighborhoods and intersections.

	First we fix some notation. Let $(X,d)$ be a metric space. We denote by
\begin{equation*}
	B(x_0,r) = \{x \in X : d(x,x_0) \leq r \}
\end{equation*}
the closed ball of radius $r$ with center in $x_0$. For any subset $A \subset X$ let
\begin{equation*}
	B(A,r)= \{x \in X : d(x,A):= \inf_{y \in A} d(x,y) \leq r \}
\end{equation*}
be the closed $r$-neighborhood of $A$.

	We call a non-empty subset of a metric space \emph{admissible} if it can be written as an intersection of balls $A = \bigcap_i B(x_i,r_i)$. Furthermore, we denote by $\mathcal{A}(X), \mathcal{E}(X), \mathcal{W}(X)$ and $\mathcal{H}(X)$ the collection of all admissible, externally hyperconvex, weakly externally hyperconvex and hyperconvex subsets of $X$.
	We always have $\mathcal{E}(X) \subset \mathcal{W}(X)\subset \mathcal{H}(X)$. Moreover, it holds that $\mathcal{A}(X) \subset \mathcal{E}(X)$ if and only if $X$ is hyperconvex.

Externally hyperconvex subspaces have the following important intersection property proven in \cite{Mie}:

\begin{Prop}\label{prop:intersection of EH}\cite[Proposition~1.2]{Mie}.
	Let $(X,d)$ be a hyperconvex space and $\{A_i\}_{i \in I}$ a family of pairwise intersecting externally hyperconvex subsets such that one of them is bounded. Then $\bigcap_{i \in I} A_i$ is nonempty and externally hyperconvex.
\end{Prop}

\begin{Lem}\label{lem:properties nbhds}
Let $X$ be a hyperconvex metric space. 
\begin{enumerate}[$(i)$]
	\item Let $A = \bigcap_{i \in I} B(x_i,r_i) \in \mathcal{A}(X)$ and $s \ge 0$. Then one has
	\[
		B(A,s) = \bigcap_{i \in I} B(x_i,r_i+s)\in \mathcal{A}(X).
	\]\label{it:properties nbhds i}
	\item For $A \in \mathcal{E}(X)$ and $s \geq 0$ one has $B(A,s) \in \mathcal{E}(X)$.
\end{enumerate}
\end{Lem}

\begin{Lem}\label{lem:characterization WEH}
	Let $A$ be a subset of the hyperconvex metric space $X$. Then, $A$ is weakly externally hyperconvex if and only if for every $x \in X$ and setting $s := d(x,A)$, the following hold:
	\begin{enumerate}[$(i)$]
		\item the intersection $B(x,s) \cap A$ is externally hyperconvex in $A$ and
		\item for every $y \in A$ there is some $a \in B(x,s) \cap A$ such $d(x,y)=d(x,a)+d(a,y)$.
\end{enumerate}	 
\end{Lem}

\begin{proof}
	If $A$ is weakly externally hyperconvex, $(i)$ clearly follows. Moreover, for $y \in A$, it also follows by weak external hyperconvexity of $A$ that there is some $a \in A \cap B(x,s) \cap B(y,d(x,y)-s) \neq \emptyset$ and therefore $d(x,y) \leq d(x,a) + d(a,y) \leq s + d(x,y)-s = d(x,y)$.
	
	For the converse first observe that $A$ must be hyperconvex since if $x \in \bigcap B(x_i,r_i)$ for $x_i \in A$ then $d(x_i,B(x,d(x,A))\cap A) \leq r_i$ by $(ii)$ and therefore there is also $x' \in A\cap B(x,d(x,A))\cap \bigcap B(x_i,r_i)$ by $(i)$.
	
	Now pick $x \in X$ and $r \geq s=d(x,A)$. Then by $(ii)$ we have $B(x,r)\cap A = B^A(B(x,s)\cap A, r-s)$ and therefore $B(x,r)\cap A$ is externally hyperconvex in $A$ by $(i)$ and Lemma~\ref{lem:properties nbhds}. Moreover by $(ii)$ for $x_i \in a$ with $d(x,x_i) \leq r+r_i$ we have $d(x_i, B(x,r)\cap A) \leq r_i$ and therefore $B(x,r)\cap \bigcap_i B(x_i,r_i) \cap A \neq \emptyset$ by Proposition~\ref{prop:intersection of EH}.
\end{proof}

\begin{Lem}\label{lem:retraction WEH}
	Let $X$ be a metric space, $A \in \mathcal{W}(X)$ and $s \geq 0$. Then there is an $s$-constant retraction $\rho \colon B(A,s) \to A$, i.e. $d(x,\rho(x)) \leq s$ for all $x \in B(A,s)$.
\end{Lem}

\begin{proof}
	Consider the partially ordered set
	\begin{align*}
	\mathcal{F} := \{ (B,\rho) : B \subset B(A,s) \text{ and } \rho \colon B \to A \text{ is an }s\text{-constant retraction} \}.
	\end{align*}
	By Zorn's Lemma there is some maximal element $(\tilde{B},\tilde{\rho}) \in \mathcal{F}$. Assume that there is some $x \in B(A,s) \setminus \tilde{B}$. For all $y \in \tilde{B}$ define $r_y = d(x,y)$. Then we have $$d(x,\rho(y)) \leq d(x,y) + d(y,\rho(y)) \leq r_y + s$$ and therefore since $A$ is weakly externally hyperconvex there is some $$z \in B(x,s) \cap \bigcap_{y \in \tilde{B}} B(\rho(y),r_y) \cap A.$$
	But then defining $\tilde{\rho}(x) := z$ we can extends $\tilde{\rho}$ to $\tilde{B} \cup \{x\}$ contradicting maximality of $(\tilde{B},\tilde{\rho})$. Hence we conclude $\tilde{B}=B(A,s)$.
\end{proof}

This and the first part of the following result can also be found in \cite{Esp,EKL} where a complete characterization of weakly externally hyperconvex subsets in terms of retractions is given.

\begin{Lem}\label{lem:nbhd of WEH}
	Let $A$ be a weakly externally hyperconvex subset of a hyperconvex space $X$. Then for any $s \geq 0$, the closed neighborhood $B(A,s)$ is weakly externally hyperconvex. Moreover, if for all $x \in X$ we have $B(x,d(x,A))\cap A \in \mathcal{E}(X)$, then this also holds for $B(A,s)$, i.e. $B(x,d(x,B(A,s)))\cap B(A,s) \in \mathcal{E}(X)$.
\end{Lem}

\begin{proof}
	Let $x\in X$ and $x_i \in B(A,s)$ such that $d(x_i,x_j) \leq r_i + r_j$, $d(x,x_i) \leq r + r_i$ and $d(x,B(A,s))\leq r$. By the previous lemma there is a retraction $\rho \colon B(A,s) \to A$ such that $d(y,\rho(y)) \leq s$. Then we have $d(\rho(x_i),x) \leq d(\rho(x_i),x_i) + d(x_i,x) \leq s + r + r_i$ and therefore there is some
	\[
		y \in B(x,s+r) \cap \bigcap_i B(\rho(x_i),r_i)\cap A
	\]
	with $d(x,y) \leq r+s$ and $d(x_i,y)\leq d(x_i ,\rho(x_i)) + d(\rho(x_i),y) \leq s+r_i$. Hence by hyperconvexity of $X$ there is some
	\[
		z \in B(x,r) \cap \bigcap_i B(x_i,r_i) \cap B(y,s) \subset B(x,r) \cap \bigcap_i B(x_i,r_i) \cap B(A,s)
	\]
	as required.
	
	Now, define $r=d(x,B(A,s))$. We claim that 
	\[
		B(x,r)\cap B(A,s) = B(x,r) \cap B(B(x,d(x,A))\cap A,s)
	\]
	and therefore if $B(x,d(x,A))\cap A \in \mathcal{E}(X)$ we also have $B(x,r)\cap B(A,s) \in \mathcal{E}(X)$.
	
	Indeed if $y \in B(x,r)\cap B(A,s)$ there is some $a \in A$ such that $d(y,a)=s$ and hence $d(x,a) \leq d(x,y)+d(y,a) \leq r+s = d(x,A)$, i.e. $a \in  B(x,d(x,A))\cap A$.
\end{proof}

\begin{Expl}\label{ex:nbhd of H}
	In general it is not true that the neighborhood of a hyperconvex subset is hyperconvex as well. Consider the isometric embedding $\iota \colon  \mathbb{R} \to l_\infty^3$ of the real line given by
	\begin{align*}
		\iota (t) = \begin{cases}
			(t,t,-t), & t \leq 0, \\
			(t,t,t),  & 0 \leq t \leq 10, \\
			(20-t,t,t)& t \geq 10,
		\end{cases}		
	\end{align*}
	and define $A=\iota(\mathbb{R}) \in \mathcal{H}(l_\infty^3)$. Then the two points $x=(-2,0,2), y=(8,10,12)$ are contained in $B(A,1)$ since $d(x,\iota(-1))=d(y,\iota(11))=1$. Now look at the intersection $B(x,5)\cap B(y,5)=\{z=(3,5,7)\}$. But $d(z,A) = d(z,\iota(5)) = 2$ and therefore $B(x,5) \cap B(y,5) \cap B(A,1) = \emptyset$, i.e. $B(A,1)$ is not hyperconvex, even not geodesic.
\end{Expl}

\begin{Expl}\label{ex:intersection of WEH}
The intersection property for externally hyperconvex subsets stated in Proposition~\ref{prop:intersection of EH} does not hold for weakly externally hyperconvex subsets. Indeed, consider the points $z:=(1,1)$, $w:=(-1,-1)$ and the half-space $H:=\{y \in l^2_{\infty} : y_1 - y_2 \ge 2\}$ in $l^2_{\infty}$. Clearly, $B(z,1)$, $B(w,1)$ and $H$ are all three elements of $\mathcal{W}(l^2_{\infty})$ and they are pairwise intersecting. However, $B(z,1) \cap B(w,1) \cap H = \emptyset$.
\end{Expl}

\begin{Lem}\label{lem:subset of WEH}
	Let $X$ be a metric space, $Y \in \mathcal{W}(X)$ and $A \in \mathcal{E}(Y)$. Then we have $A \in \mathcal{W}(X)$.
\end{Lem}

\begin{proof}
	Let $x \in X$, $r \geq d(x,A)$ and let $\{B(x_i,r_i)\}_{i\in I}$ be a collection of balls with $x_i \in A$ , $d(x_i,x_j) \leq r_i + r_j$ and $d(x_i,x) \leq r_i + r$. Then the sets $B_\epsilon := B(x,r+\epsilon)\cap Y$, $A_i := B(x_i,r_i)\cap Y$ and $A$ are pairwise intersecting externally hyperconvex subsets of $Y$ and therefore by Proposition~\ref{prop:intersection of EH} we have
	\[
		B(x,r) \cap A \cap \bigcap_i B(x_i,r_i) = \bigcap_{\epsilon > 0} B_\epsilon \cap A \cap \bigcap_i A_i \neq \emptyset.
	\]
\end{proof}

\begin{Lem}\label{lem:intersection of WEH}
	Let $X$ be a metric space, $A \in \mathcal{E}(X), Y \in \mathcal{W}(X)$ such that $A\cap Y \neq \emptyset$. Let $\{ x_i \}_{i \in I} \subset Y$ be a collection of points with $d(x_i,x_j) \leq r_i+r_j$ and $d(x_i,A) \leq r_i$. Then for any $s > 0$ there are $a\in A \cap \bigcap_i B(x_i,r_i)$ and $y \in Y \cap \bigcap_i B(x_i,r_i)$ with $d(a,y) \leq s$.
\end{Lem}

\begin{proof}
	Let $y_0 \in A \cap Y$ and $d=d(y_0,\bigcap B(x_i,r_i))$. Without loss of generality we may assume that $s \leq d$. Then since $A \in \mathcal{E}(X)$ and $Y \in \mathcal{W}(X)$ there are 
\begin{align*}
	a_1 &\in \bigcap_i B(x_i, r_i+d-s) \cap B(y_0,s) \cap  A, \\
	y_1 &\in \bigcap_i B(x_i, r_i+d-s) \cap B(a_1,s) \cap Y.
\end{align*}
Proceeding this way, we can choose inductively
\begin{align*}
	a_n &\in \bigcap_i B(x_i, r_i+d-ns) \cap B(y_{n-1},s) \cap A, \\
	y_n &\in \bigcap_i B(x_i, r_i+d-ns) \cap B(a_n,s) \cap Y
\end{align*}
	for $n \leq \lfloor \frac{d}{s} \rfloor =: n_0$ and finally there are
\begin{align*}
	a &\in \bigcap_i B(x_i, r_i) \cap B(y_{n_0},s) \cap A, \\
	y &\in \bigcap_i B(x_i, r_i) \cap B(a,s) \cap Y
\end{align*}
	as desired.
\end{proof}

\begin{Prop}\label{prop:intersection of WEH}
	Let $X$ be a metric space, $A \in \mathcal{E}(X), Y \in \mathcal{W}(X)$ such that $A\cap Y \neq \emptyset$. Then we have $A\cap Y \in \mathcal{E}(Y)$ and therefore $A\cap Y \in \mathcal{W}(X)$.
\end{Prop}

\begin{proof}
	Let $\{ x_i \}_{i \in I} \subset Y$ with $d(x_i,x_j) \leq r_i+r_j$ and $d(x_i,A) \leq r_i$. We now construct inductively two converging sequences $(a_n) \subset \bigcap_i B(x_i,r_i) \cap A$ and $(y_n) \subset \bigcap_i B(x_i,r_i) \cap Y$ wirh $d(a_n,y_n) \leq \frac{1}{2^{n+1}}$ as follows. By the previous lemma we may choose
	\begin{align*}
		a_0 &\in \bigcap_i B(x_i,r_i) \cap A, \\
		y_0 &\in \bigcap_i B(x_i,r_i) \cap Y
	\end{align*}
	with $d(a_0,y_0) \leq \frac{1}{2}$. Assume now that we have
	\begin{align*}
		a_n &\in \bigcap_i B(x_i,r_i) \cap A, \\
		y_n &\in \bigcap_i B(x_i,r_i) \cap Y
	\end{align*}
	with $d(a_n,y_n) \leq \frac{1}{2^{n+1}}$. Then applying Lemma~\ref{lem:intersection of WEH} for the collection of balls $\{B(x_i,r_i) \}_{i \in I} \cup \{ B(y_n,\frac{1}{2^{n+1}})\}$ we find
	\begin{align*}
		a_{n+1} &\in \bigcap_i B(x_i,r_i) \cap B(y_n,\tfrac{1}{2^{n+1}}) \cap A, \\
		y_{n+1} &\in \bigcap_i B(x_i,r_i) \cap B(y_n,\tfrac{1}{2^{n+1}}) \cap Y
	\end{align*}
	with $d(a_{n+1},y_{n+1}) \leq \frac{1}{2^{n+2}}$. Especially we have $d(y_{n+1},y_n) \leq \frac{1}{2^{n+1}}$ and
	\[
		d(a_{n+1},a_n) \leq d(a_{n+1},y_n)+d(y_n,a_n) \leq \frac{1}{2^{n+1}}+\frac{1}{2^{n+1}}=\frac{1}{2^n}.
	\]
	Hence the two sequences are Cauchy and therefore converge to some common limit point $z \in A\cap Y \cap \bigcap_i B(x_i,r_i) \neq \emptyset$. Finally $A\cap Y \in \mathcal{W}(X)$ by Lemma~\ref{lem:subset of WEH}.
\end{proof}

Looking at the proof carefully, we see that only $d(x_i,A) \leq r_i$ is assumed. Therefore we may deduce the following corollary:

\begin{Cor}\label{cor:distance in WEH}
	Let $X$ be a metric space, $A \in \mathcal{E}(X), Y \in \mathcal{W}(X)$ such that $A\cap Y \neq \emptyset$. Then for all $x \in Y$ we have $d(x,A)=d(x, A \cap Y)$.
\end{Cor}

\begin{Cor}\label{cor:transitivity of WEH}
	Let $X$ be a metric space, $A \in \mathcal{W}(Y)$ for $Y \in \mathcal{W}(X)$. Then we have $A \in \mathcal{W}(X)$.
\end{Cor}

\begin{proof}
	Let $x\in X$ with $d(x,A) \leq r$ and $\{ x_i \} _{i \in I} \subset A$ with $d(x_i,x_j) \leq r_i + r_j, d(x_i,x) \leq r_i + r$. Then we have $B(x,r) \cap Y \in \mathcal{E}(Y)$ and $B(x,r) \cap A = (B(x,r) \cap Y) \cap A \in \mathcal{E}(A)$ by Proposition~\ref{prop:intersection of WEH}.
	Moreover, by applying Corollary~\ref{cor:distance in WEH} twice we have
	\[
		d(x_i,(B(x,r) \cap Y) \cap A) = d(x_i,B(x,r) \cap Y) = d(x_i,B(x,r)) \leq r_i
	\]
	and hence $\bigcap_i B(x_i,r_i) \cap B(x,r) \cap A \neq \emptyset$.
\end{proof}

\begin{Cor}\label{cor:admitting distance}
	Let $X$ be a hyperconvex metric space, $A \in \mathcal{E}(X), Y \in \mathcal{W}(X)$. Then there are $a\in A, y\in Y$ with $d(a,y)=d(A,Y)$
\end{Cor}

\begin{proof}
	Let $s=d(A,Y)$. First observe that for any $n \in \mathbb{N}$ we have $B(A,s+\frac{1}{2^{n}}) \in \mathcal{E}(X)$ and $B(A,s+\frac{1}{2^{n}})\cap Y \in \mathcal{W}(X)$ by Proposition~\ref{prop:intersection of WEH}.
	So given $y_n \in B(A,s+\frac{1}{2^n}) \cap Y$ by Corollary~\ref{cor:distance in WEH} there is some $y_{n+1} \in B(y_n, \frac{1}{2^{n+1}}) \cap B(A, s+ \frac{1}{2^{n+1}}) \cap Y$. Therefore we can find a sequence $(y_n) \subset Y$ with $d(y_n,A) \leq s+\frac{1}{2^n}$ and $d(y_n,y_{n-1})\leq \frac{1}{2^n}$. Hence $(y_n)$ is Cauchy and converges to some point $y \in Y$ with $d(y,A) \leq s$. Now since $A$ is proximinal there is some $a \in A$ with $d(a,y) \leq s=d(A,Y)$ as required.
\end{proof}

The following proposition answers an open question on the intersection of weakly externally hyperconvex sets stated in \cite{EspK} for \textit{proper} metric spaces, i.e. for spaces where all closed balls are compact.

\begin{Prop}\label{prop:IntersectionWEH}
	Let $X$ be a proper hyperconvex metric space and $Y,Y' \in \mathcal{W}(X)$ with non-empty intersection. Then we have $Y \cap Y' \in \mathcal{W}(X)$.
\end{Prop}

\begin{proof}
	By Corollary~\ref{cor:transitivity of WEH} it is enough to show that $Y\cap Y' \in \mathcal{W}(Y)$. Therefore let $\{ B(x_i,r_i)\}_{i \in I}$ be a collection of balls with $x_i \in Y \cap Y'$ and $d(x_i,x_j) \leq r_i + r_j$ and $x \in Y$ with $d(x, Y\cap Y') \leq r$ and $d(x_i,x) \leq r_i + r$.
	Let $s > 0$. Since $d(x, Y\cap Y') \leq r$ there is some $y_0 \in Y \cap Y' \cap B(x,r+s).$
	Define $d=d(y_0,\bigcap_i B(x_i,r_i))$. Then there is some 
	\[
	y_0' \in B(y_0,s) \cap Y' \cap B(x,r) \cap \bigcap_i B(x_i,r_i+d).
	\]
	Now for $n \leq \lfloor \frac{d}{s} \rfloor =: n_0$, we can choose inductively
	\begin{align*}
	y_n &\in B(y_{n-1}',s) \cap Y \cap \bigcap_i B(x_i,r_i + d - ns)  \text{ and } \\
	y_n' &\in B(y_n,s) \cap Y' \cap B(x,r) \cap \bigcap_i B(x_i,r_i + d - ns).
	\end{align*}
	Finally there are
	\begin{align*}
	y &\in B(y_{n_0}',s) \cap Y \cap \bigcap_i B(x_i,r_i)  \text{ and } \\
	y' &\in B(y,s) \cap Y' \cap B(x,r) \cap \bigcap_i B(x_i,r_i)
	\end{align*}
	and hence $d(Y \cap \bigcap_i B(x_i,r_i),Y' \cap \bigcap_i B(x_i,r_i) \cap B(x,r)) \leq d(y,y') \leq s$, i.e. $d(Y \cap \bigcap_i B(x_i,r_i),Y' \cap \bigcap_i B(x_i,r_i) \cap B(x,r)) = 0$ and 
	since $X$ is proper, both sets are compact and therefore	
	their intersection $Y \cap Y' \cap \bigcap_i B(x_i,r_i) \cap B(x,r)$ is non-empty.
\end{proof}

\begin{Prop}\label{Prop:IncreasingSequenceWEH}
Let $X$ be any metric space and let $(Y_n)_{n \in \mathbb{N}} \subset \mathcal{W}(X)$ be an increasing sequence (for the inclusion) such that $Y := \overline{\bigcup_n Y_n}$ is proper. Then, one has $Y \in \mathcal{W}(X)$.
\end{Prop}
\begin{proof}
Consider a family $\{ (x_i,r_i ) \}_{i \in I}$ in $Y \times \R$ such that $d(x_i,x_j) \leq r_i + r_j$. Moreover, let $(x,r) \in X \times \R$ satisfy $d(x,Y) \le r$ and $d(x,x_i) \leq r + r_i$. 

There is a decreasing sequence $s_n \downarrow 0$ such that $d(x,Y_n) \leq r + s_n$. 
Now fix $\epsilon = \frac{1}{m} > 0$. Then for every $i \in I$ there is some $y_i \in B(x_i, \epsilon) \cap \bigcap_n Y_n$. For $n \in \mathbb{N}$ let
\[
I_n := \bigl \{ i \in I : y_i \in Y_n \bigr \}
\]
and since $Y_n$ is weakly externally hyperconvex there is some
\[ 
z_n \in Y_n \cap B(x,r + s_n) \cap \bigcap_{i \in I_n} B(y_i,r_i+\epsilon).
\]
Since $Y$ is proper and $(z_n) \subset Y \cap B(x,r+s_0)$, it follows that there is a convergent subsequence $z_{n_k} \to z^m \in Y \cap B(x,r) \cap \bigcap_{i \in I} B(x_i,r_i+\frac{1}{m})$. Moreover, since $Y$ is proper there is a subsequence $z^{m_l} \to z \in Y \cap B(x,r) \cap \bigcap_{i \in I} B(x_i,r_i)$. This proves that $Y$ is weakly externally hyperconvex in $X$.
\end{proof}

\begin{Lem}\label{lem:locally externally hyperconvex}
	Let $X$ be a hyperconvex metric space and $A \in \mathcal{W}(X)$. Assume that there is some $s > 0$ such that $A \in \mathcal{E}(B(A,s))$. Then $A \in \mathcal{E}(X)$.
\end{Lem}

\begin{proof}
	First we show that $A \in \mathcal{E}(B(A,r))$ for any $r \geq 0$. By assumption this holds for $r=s > 0$. Therefore it is enough to prove $A \in \mathcal{E}(B(A,r)) \Rightarrow A \in \mathcal{E}(B(A,2r))$. Let $(x_i,r_i)_{i \in I} \in B(A,2r) \times \mathbb{R}_{\geq 0}$ with $d(x_i,x_j) \leq r_i + r_j$ and $d(x_i,A) \leq r_i$. Define $A_i = B(x_i,r_i) \cap B(A,r) \in \mathcal{E}(B(A,r))$. Clearly $A \cap A_i \neq \emptyset$.	
	By Lemma~\ref{lem:nbhd of WEH} $B(A,r) \in \mathcal{W}(X)$ and hence by Lemma~\ref{lem:retraction WEH} there is a retraction $\rho \colon B(A,2r) \to B(A,r)$ with $d(\rho(x),x) \leq r$. Set $y_i = \rho (x_i)$ and since $A \in \mathcal{E}(B(A,r))$ there is some $z \in A \cap \bigcap_i B(y_i,r_i)$ with $d(x_i,z) \leq r_i + r$. Therefore since $X$ is hyperconvex we get $\emptyset \neq B(z,r)\cap \bigcap_i B(x_i,r_i) \subset B(A,r)$. Especially $A_i\cap A_j = B(x_i,r_i) \cap B(x_j,r_j) \cap B(A,r) \neq \emptyset$ and hence $A \cap \bigcap_i B(x_i,r_i) = A \cap \bigcap_i A_i \neq \emptyset$ by Proposition~\ref{prop:intersection of EH}.
	
	To conclude that $A \in \mathcal{E}(X)$, let $(x_i,r_i)_{i \in I} \in X \times \mathbb{R}_{\geq 0}$ with $d(x_i,x_j) \leq r_i + r_j$ and $d(x_i,A) \leq r_i$. Define $A_i = B(x_i,r_i) \cap A \in \mathcal{E}(A)$. For fixed $i,j \in I$ we have $x_i,x_j \in B(A,r)$ for some $r \geq 0$ and hence by the first step we have $A_i \cap A_j = A \cap B(x_i,r_i) \cap B(x_j,r_j) \neq \emptyset$. Therefore we get $A \cap \bigcap_i B(x_i,r_i) = \bigcap_i A_i \neq \emptyset$.	
\end{proof}

\begin{Lem}\label{Lem:Products}
	Let $X = X^1 \times_\infty X^2$ be the product of two metric spaces with $d(x,y) = \max_{\lambda = 1,2} d_\lambda(x^\lambda,y^\lambda)$. Moreover let $A = A^1 \times A^2$ be a subset of $X$. Then the following holds:
	\begin{enumerate}[(i)]
	\item $B(x,r) = B^1(x^1,r) \times B^2(x^2,r)$,
	\item $A \in \mathcal{E}(X) \Leftrightarrow A^\lambda \in \mathcal{E}(X^\lambda)$ for $\lambda = 1,2$,	
	\item $A \in \mathcal{W}(X) \Leftrightarrow A^\lambda \in \mathcal{W}(X^\lambda)$ for $\lambda = 1,2$.
	\item $X$ is hyperconvex $\Leftrightarrow$ $X^\lambda$ is hyperconvex for $\lambda = 1,2$.
	\end{enumerate}
\end{Lem}

\begin{proof}
	Property (i) follows from the fact that $d(x,y) \leq r$ if and only if $d_\lambda(x^\lambda,y^\lambda) \leq r$ for $\lambda = 1,2$. For (ii) let first $x_i \in X$ be any collection of points with $d(x_i,x_j) \leq r_i + r_j$ and $d(x_i,A) \leq r_i$. Then $d(x_i^\lambda,x_j^\lambda) \leq d(x_i,x_j) \leq r_i + r_j$ and $d(x_i^\lambda,A^\lambda) \leq d(x_i,A) \leq r_i$ and therefore if $A^\lambda \in \mathcal{E}(X^\lambda)$ there is some $y^\lambda \in A^\lambda \cap \bigcap_i B^\lambda(x_i^\lambda,r_i)$ and hence $y=(y^1,y^2) \in A \cap \bigcap_i B(x_i,r_i)$. For the converse if $d(x_i^1,x_j^1) \leq r_i + r_j$ and $d(x_i^1,A^1) \leq r_i$ fix some $x^2 \in A^2$. Then the points $x_i=(x_i^1,x^2)$ fulfill $d(x_i,x_j) \leq r_i + r_j$ and $d(x_i,A) \leq r_i$, i.e. there is some $y = (y^1,y^2) \in A \cap \bigcap_i B(x_i,r_i)$ and hence $y^1 \in A^1 \cap \bigcap_i B^1(x_i^1,r_i) \neq \emptyset$. The proof of (iii) is similar and (iv) follows from (ii) by setting $A^\lambda = X^\lambda$.
\end{proof}


\section{Gluing along Weakly Externally Hyperconvex Subspaces}


\begin{Def}\label{def:gluing}
Let $(X_\lambda,d_\lambda)_{\lambda\in\Lambda}$ be a family of metric spaces with closed subsets $A_\lambda\subset X_\lambda$. Suppose that all $A_\lambda$ are isometric to some metric space $A$. For every $\lambda\in\Lambda$ fix an isometry $\varphi_\lambda \colon A \to A_\lambda$. We define an equivalence relation on the disjoint union $\bigsqcup_\lambda X_\lambda$ generated by $\varphi_\lambda(a) \sim \varphi_{\lambda'}(a)$ for $a \in A$. The resulting space $X := (\bigsqcup_\lambda X_\lambda)/\sim$ is called the \emph{gluing} of the $X_\lambda$ along $A$.
\end{Def}

$X$ admits a natural metric. For $x \in X_\lambda$ and $y\in X_{\lambda'}$ it is given by
	\begin{equation}
		d(x,y) = \begin{cases}
				d_\lambda(x,y), &\text{ if } \lambda=\lambda', \\
				\inf_{a \in A} \{ d_\lambda(x,\varphi_\lambda(a)) + d_{\lambda'}(\varphi_{\lambda'}(a),y) \}, &\text{ if } \lambda \neq \lambda'.
				\end{cases}
	\end{equation}
For more details see for instance \cite[Lemma~I.5.24]{Bri}.

If there is no ambiguity, indices for $d_\lambda$ are dropped and the sets $A_\lambda=\varphi_\lambda(A) \subset X_\lambda$ are identified with $A$.

Balls inside the subset $X_\lambda$ are denoted by $B^\lambda(x,r)$.

\begin{Lem}\label{lem:list properties gluing}
Let $X$ be a hyperconvex metric space obtained by gluing a family of hyperconvex metric spaces $(X_\lambda,d_\lambda)_{\lambda \in \Lambda}$ along some set $A$. Then $A$ is hyperconvex.
\end{Lem}

\begin{proof}
	Let $x_i \in A$ such that $d(x_i,x_j) \leq r_i +r_j$. Then for each $\lambda$, since $X_\lambda$ is hyperconvex there is some $y_\lambda \in \bigcap B(x_i,r_i) \cap X_\lambda$. Moreover $\bigcap B(x_i,r_i)$ is path-connected and a path from $y_\lambda$ to $y_{\lambda'}$ must intersect $A$, i.e. $\bigcap B(x_i,r_i)\cap A \neq \emptyset$.
\end{proof}

In general we cannot say more about necessary conditions on $A$ such that the gluing along $A$ is hyperconvex. For instance gluing a hyperconvex space $X$ and any hyperconvex subset $A \in \mathcal{H}(X)$ along $A$ the resulting space is isometric to $X$ and therefore hyperconvex. But there are also plenty of non-trivial examples.

\begin{Expl}
	Let $f \colon \mathbb{R} \to \mathbb{R}$ be any 1-lipschitz function. Consider its graph $A = \{ (x,y) \in l_\infty^2 : y=f(x)\}$ and the two sets $X_1 = \{ (x,y) \in l_\infty^2 : y \leq f(x)\}$, $X_2 = \{ (x,y) \in l_\infty^2 : y \geq f(x)\}$. Then $l_\infty^2 = X_1 \sqcup_A X_2$ is hyperconvex and occurs as the gluing of two hyperconvex spaces $X_1,X_2$.
\end{Expl}

But if we assume that the gluing set is weakly externally hyperconvex, we can do better.

\begin{Lem}\label{lem:distance}
	Let $(X,d)$ be the metric space obtained by gluing a family of hyperconvex metric spaces $(X_\lambda,d_\lambda)_{\lambda \in \Lambda}$ along some set $A$ such that $A$ is weakly externally hyperconvex in $X_\lambda$ for each $\lambda$. For $x\in X_\lambda$ and $x'\in X_{\lambda'}$, there are then points $a\in B(x,d(x,A))\cap A$ and $a' \in B(x',d(x',A))\cap A$ such that $$d(x,x')=d(x,a)+d(a,a')+d(a',x').$$
\end{Lem}

\begin{proof}
	Since $A$ is weakly externally hyperconvex in each $X_\lambda$, by Lemma \ref{lem:characterization WEH} $(ii)$, there are for every $y \in A$, points $a\in B(x,d(x,A))\cap A$ and $a' \in B(x',d(x',A))\cap A$ such that both $d(x,y)=d(x,a)+d(a,y)$ and $d(y,x')=d(y,a')+d(a',x')$ hold. Hence,
\[
	d(x,x') = d(x,A) + d(B(x,d(x,A))\cap A,B(x',d(x',A))\cap A)+d(x',A).
\]
	But the sets $B(x,d(x,A))\cap A$ and $B(x',d(x',A))\cap A$ are externally hyperconvex in $A$ and therefore by Lemma~\ref{cor:admitting distance} there are $a,a' \in A$ with
\[
	d(a,a') = d(B(x,d(x,A))\cap A,B(x',d(x',A))\cap A).
\]
\end{proof}

\begin{Lem}\label{lem:balls}
	Let $(X,d)$ be the metric space obtained by gluing a family of hyperconvex metric spaces $(X_\lambda,d_\lambda)_{\lambda \in \Lambda}$ along some set $A$ such that $A$ is weakly externally hyperconvex in $X_\lambda$ for each $\lambda$. Then, for $\lambda\neq\lambda'$, $x\in X_\lambda$ and $r \geq s:=d(x,A)$, one has
\[
	B(x,r) \cap X_{\lambda'} = B^{\lambda'}(B^\lambda(x,s)\cap A,r-s).
\]
	Therefore, if $B^\lambda(x,s)\cap A$ is externally hyperconvex in $X_{\lambda'}$, then so is $B(x,r) \cap X_{\lambda'}$.
\end{Lem}

\begin{proof}
	Let $x' \in B(x,r) \cap X_{\lambda'}$. By Lemma~\ref{lem:distance}, there is some $a \in B^\lambda(x,s)\cap A$ with $d(x,x')=d(x,a)+d(a,x')$ . We have $d(a,x') \leq r-s$ and hence
\[
	x' \in B^{\lambda'}(B^\lambda(x,s)\cap A,r-s).
\]
\end{proof}

\begin{Prop}\label{prop:externally hyperconvex balls}
	Let $(X,d)$ be the metric space obtained by gluing a family of hyperconvex metric spaces $(X_\lambda,d_\lambda)_{\lambda \in \Lambda}$ along some set $A$ such that $A$ is weakly externally hyperconvex in $X_\lambda$ for each $\lambda$. If $X$ is hyperconvex then for all $\lambda \in \Lambda$ and all $x \in X \setminus X_\lambda$ the set $B(x,d(x,A))\cap A$ is externally hyperconvex in $X_\lambda$.
\end{Prop}

\begin{proof}
	Set $s:=d(x,A)$ and let $\{x_i\}_{i \in I}$ be a collection of point in $X_\lambda$ and $\{r_i\}_{i \in I}$ such that $d(x_i,x_j) \leq r_i+r_j$ and $d(x_i, B(x,s)\cap A) \leq r_i$. Then, by hyperconvexity of $X$ there is some $y \in B(x,s) \cap \bigcap B(x_i,r_i)$. Since $y\in B(x,s)$, we have $y\in X_{\lambda'}$ for some $\lambda' \neq \lambda$. Therefore, by Lemma~\ref{lem:distance} there is for each $i$ some $y_i \in B(x_i,d(x_i,A))\cap A$ with $d(y,x_i)=d(y,y_i)+d(y_i,x_i)$. Define $r_i' = r_i-d(y_i,x_i)$. We have $d(y_i,y_j) \leq d(y_i,y)+d(y,y_j) \leq r_i' + r_j'$ and $d(x,y_i) \leq s+r_i'$. Hence, since $A$ is weakly externally hyperconvex in $X_{\lambda'}$, there is some $z \in \bigcap_i B(y_i,r_i') \cap B(x,s) \cap A$ and thus $\bigcap_i B(x_i,r_i) \cap B(x,s) \cap A \neq \emptyset$.
\end{proof}

\begin{Prop}\label{prop:gluing along WEH}
	Let $(X,d)$ be the metric space obtained by gluing a family of hyperconvex metric spaces $(X_\lambda,d_\lambda)_{\lambda \in \Lambda}$ along some set $A$ such that for each $\lambda$ the set $A$ is weakly externally hyperconvex in $X_\lambda$ and for any $x\in X \setminus X_\lambda$ the intersection $B(x,d(x,A))\cap A$ is externally hyperconvex in $X_\lambda$. Then $X$ is hyperconvex and $X_{\lambda} \in \mathcal{W}(X)$ for every $\lambda$.
\end{Prop}

Note that this proposition generalizes the results for gluings along strongly convex and along externally hyperconvex subsets in \cite{Mie}. Clearly, in both cases the gluing set is weakly externally hyperconvex. In the first case, the intersection $B(x,d(x,A))\cap A$ is a single point (the gate) and therefore externally hyperconvex in $X$. In the second case, we have $B(x,d(x,A))\cap A \in \mathcal{E}(A)$ and therefore by Proposition~4.9 in \cite{Mie} we get $B(x,d(x,A))\cap A \in \mathcal{E}(X_\lambda)$.

\begin{proof}
Let $\{ B(x_i,r_i) \}_{i \in I}$ be a family of balls with $d(x_i,x_j) \leq r_i + r_j$. We divide the proof into two cases. 
\\ \\
\noindent \textit{Case 1.}
If for every $i,j \in I$, one has 
\[
B(x_i,r_i) \cap B(x_j,r_j)\cap A \neq \emptyset,
\]
setting $C_i := A \cap B(x_i,r_i)$, we obtain that the family $\{ C_i\}_{i \in I}$ is pairwise intersecting. Moreover, $\{ C_i\}_{i \in I}$ is contained in $\mathcal{E}(A)$ since $A$ is weakly externally hyperconvex. By Proposition~\ref{prop:intersection of EH} we obtain that $\bigcap_{i \in I} C_i \neq \emptyset$, hence $\bigcap_{i \in I} B(x_i,r_i) \neq \emptyset$.
\\ \\
\noindent \textit{Case 2.}
Otherwise there are $i_0,j_0 \in I$ with $x_{i_0},x_{j_0} \in X_{\lambda_0}$ such that 
\[
B(x_{i_0},r_{i_0}) \cap B(x_{j_0},r_{j_0})\cap A = \emptyset.
\]
Indeed, either there is some $i_0 \in I$ such that $d(x_{i_0},A) > r_{i_0}$ and we may take $i_0=j_0$ or if 
\[
B(x_{i_0},r_{i_0}) \cap B(x_{j_0},r_{j_0})\cap A = \emptyset
\]
with $d(x_{i_0},A) \leq r_{i_0}$ and $d(x_{j_0},A) \leq r_{j_0}$, we get $x_{i_0},x_{j_0} \in X_{\lambda_0}$ by Lemma \ref{lem:distance}. Observe that in both cases we may assume that for $x_i \in X_\lambda \neq X_{\lambda_0}$ we have $d(x_i,A) \leq r_i$.

Define $A_i^{\lambda_0} = B(x_i,r_i) \cap X_{\lambda_0}.$
The goal is now to show the following claim:
\\ \\
\noindent \textit{Claim. For every $i,j \in I$, one has $A_i^{\lambda_0} \cap A_j^{\lambda_0} \neq \emptyset$.}
\\ \\
Then by Lemma~\ref{lem:balls} we have $A_i^{\lambda_0} \in \mathcal{E}(X_{\lambda_0})$ and by Proposition~\ref{prop:intersection of EH} we get $$\bigcap B(x_i,r_i) \cap X_{\lambda_0} = \bigcap A_i^{\lambda_0} \neq \emptyset.$$

To prove the claim consider first the following two easy cases.
\begin{itemize}
\item If $x_i,x_j \in X_{\lambda_0}$, then we are done by hyperconvexity of $X_{\lambda_0}$.
\item If $x_i \in X_{\lambda} \neq X_{\lambda'} \ni x_j$, we have $B(x_i,r_i) \cap  B(x_j,r_j) \cap X_{\lambda_0} \neq \emptyset$ by Lemma \ref{lem:distance} and we are done.
\end{itemize}

The remaining case is when $x_i,x_j \in X_{\lambda} \neq X_{\lambda_0}$. We do this in two steps.
\\ \\
\noindent \textit{Step I.}
Set 
\[
A' := B(x_{i_0},r_{i_0}) \cap B(x_{j_0},r_{j_0}) = B^{\lambda_0}(x_{i_0},r_{i_0}) \cap B^{\lambda_0}(x_{j_0},r_{j_0})
\]
and $s := d(A,A')$. By Corollary~\ref{cor:admitting distance} we have $B(A',s)\cap A \neq \emptyset$. Furthermore we get
\[
B(A',s) = B(x_{i_0},r_{i_0}+s) \cap B(x_{j_0},r_{j_0}+s)
\]
and hence $B(x_{i_0},r_{i_0}+s) \cap B(x_{j_0},r_{j_0}+s) \subset X_{\lambda_0}$.

To see this, observe first that by \eqref{it:properties nbhds i} in Lemma~\ref{lem:properties nbhds}, we have
\[
	B(A',s) = B^{\lambda_0}(x_{i_0},r_{i_0}+s) \cap B^{\lambda_0}(x_{j_0},r_{j_0}+s) \subset X_{\lambda_0}
\]
and therefore
\[
	\left( B(x_{i_0},r_{i_0}+s) \cap B(x_{j_0},r_{j_0}+s) \right) \setminus B(A',s)  \subset X \setminus X_{\lambda_0}.
\]
Thus assume that there is some $$y \in B(x_{i_0},r_{i_0}+s) \cap B(x_{j_0},r_{j_0}+s) \cap \left(X_\lambda \setminus X_{\lambda_0}\right).$$ Now since $$B(x_{i_0},r_{i_0}+s) \cap B(x_{j_0},r_{j_0}+s) \cap A \neq \emptyset$$ and $$B(x_{i_0},r_{i_0}+s) \cap B(x_{j_0},r_{j_0}+s) \cap X_\lambda$$ is externally hyperconvex in $X_\lambda$ and thus path-connected, there is some $$y' \in B(x_{i_0},r_{i_0}+s) \cap B(x_{j_0},r_{j_0}+s) \cap X_\lambda \setminus X_{\lambda_0}$$ with $d(y',A) \leq s$. But then by Lemma~\ref{lem:distance} we have
\begin{align*}
B(x_{i_0},r_{i_0}) \cap B(y',s) \cap X_{\lambda_0} &\neq \emptyset, \\
B(x_{j_0},r_{j_0}) \cap B(y',s) \cap X_{\lambda_0} &\neq \emptyset
\end{align*}
and therefore $$B(x_{i_0},r_{i_0}) \cap B(x_{j_0},r_{j_0}) \cap B(y',s) \cap X_{\lambda_0} \neq \emptyset,$$ i.e. $y' \in B(A',s)$ contradicting $y' \notin X_{\lambda_0}$.
\\ \\
\noindent \textit{Step II.}
We now show that the family
\[
\mathcal{F} :=\{ B(x_{i_0},r_{i_0}+s)\cap X_\lambda, B(x_{j_0},r_{j_0}+s)\cap X_\lambda, B^{\lambda}(x_i,r_i), B^{\lambda}(x_j,r_j) \}
\]
is pairwise intersecting. We already observed that
\[
(B(x_{i_0},r_{i_0}+s)\cap X_\lambda) \cap (B(x_{j_0},r_{j_0}+s)\cap X_\lambda) \neq \emptyset.
\]
Further, since $x_{i_0} \in X_{\lambda_0} \neq X_{\lambda} \ni x_i$, by Lemma \ref{lem:distance}, one has
\[
(B(x_{i_0},r_{i_0}+s)\cap X_\lambda) \cap B^{\lambda}(x_i,r_i) \neq \emptyset
\]
and similarly for $(i_0,i)$ replaced by $(i_0,j)$ as well as by $(j_0,i)$ and $(j_0,j)$. Finally, 
\[
B^{\lambda}(x_i,r_i) \cap B^{\lambda}(x_j,r_j) \cap X_{\lambda} \neq \emptyset
\]
by hyperconvexity of $X_{\lambda}$. Hence, we have shown that $\mathcal{F}$ is pairwise intersecting. Since $\mathcal{F} \subset \mathcal{E}(X_{\lambda})$ it follows by Proposition~\ref{prop:intersection of EH} that
\[
C:= B^{\lambda}(x_{i_0},r_{i_0}+s) \cap  B^{\lambda}(x_{j_0},r_{j_0}+s) \cap  B^{\lambda}(x_i,r_i) \cap  B^{\lambda}(x_i,r_i) \neq \emptyset.
\]
Since $B(x_{i_0},r_{i_0}+s) \cap  B(x_{j_0},r_{j_0}+s) \cap X_\lambda \subset A$ we have in particular $C \subset A$. Hence
\[
B^{\lambda}(x_i,r_i) \cap  B^{\lambda}(x_j,r_j) \cap A \supset C \cap A = C \neq \emptyset,
\]
and this is the desired result.
\\

To see that $X_\lambda$ is weakly externally hyperconvex in $X$, use that for $x \in X$, $r \geq d(x,X_\lambda)$, $x_i \in X_\lambda$ with $d(x,x_i) \leq r + r_i$, $d(x_i,x_j) \leq r_i + r_j$ we have $B(x,r) \cap B(x_i,r_i) \cap X_\lambda \neq \emptyset$ by Lemma~\ref{lem:distance} and therefore $\{ B(x,r) \cap X_\lambda, B^\lambda(x_i,r_i) \}$ is a family of pairwise intersecting externally hyperconvex subsets of $X_\lambda$.
\end{proof}

Combining Propositions~\ref{prop:externally hyperconvex balls} and \ref{prop:gluing along WEH} we get Theorem~\ref{thm:gluing along WEH}.

\begin{Thm}
	Let $X_0$ be a hyperconvex metric space and $\{ X_\lambda \}_{\lambda \in \Lambda}$ a family of hyperconvex metric spaces with weakly externally hyperconvex subsets $A_\lambda \in \mathcal{W}(X_\lambda)$ such that for every $\lambda$ there is an isometric copy $A_\lambda \in \mathcal{W}(X_0)$ and $A_\lambda \cap A_{\lambda'}= \emptyset$ for $\lambda \neq \lambda'$. If for every $x_\lambda \in X_\lambda$ and every $x \in X_0$ we have $B(x_\lambda,d(x_\lambda,A_\lambda))\cap A_\lambda \in \mathcal{E}(X_0)$ and $B(x,d(x,A_\lambda))\cap A_\lambda \in \mathcal{E}(X_\lambda)$, then $X = X_0 \bigsqcup_{\{A_\lambda : \lambda \in \Lambda\}} X_\lambda$ is hyperconvex.
\end{Thm}

\begin{proof}
	First by Theorem~\ref{thm:gluing along WEH} we get that $Y_\lambda = X_0 \sqcup_{A_\lambda} X_\lambda$ is hyperconvex and $X_0 \in \mathcal{W}(Y_\lambda)$. Observe that $X$ can be obtained by gluing the spaces $Y_\lambda$ along $X_0$, i.e. $X= \bigsqcup_{X_0} Y_\lambda$. Therefore it remains to prove that for $\lambda \neq \lambda'$ and $x \in Y_\lambda$ the intersection $B:=B(x,d(x,X_0)) \cap X_0 \in \mathcal{E}(Y_{\lambda'})$.

	By Corollary~\ref{cor:transitivity of WEH} clearly we have $B \in \mathcal{W}(Y_{\lambda'})$. Without loss of generality we may assume that $x \notin X_0$. Then we have $d(x,X_0)=d(x,A_\lambda)$ and therefore $B=B(x,d(x,A_\lambda))\cap A_\lambda \in \mathcal{E}(X_0)$, especially $B \subset A_\lambda$. Hence by Corollary~\ref{cor:admitting distance} we get $d(B,A_{\lambda'}) > 0$, i.e. there is some $s > 0$ such that $B^{\lambda'}(B,s) \subset X_0$. Thus $B \in \mathcal{E}( B^{\lambda'}(B,s))$ and by Lemma~\ref{lem:locally externally hyperconvex} we get $B \in \mathcal{E}(Y_{\lambda'})$ as desired.	
\end{proof}

\begin{Prop}\label{prop:gluing copies}
	Let $X$ be a metric space and $A \subset X$ such that $X \sqcup_A X$ is hyperconvex. Then, the following hold
	\begin{enumerate}[(i)]
	\item $A$ is weakly externally hyperconvex in $X$,
	\item For every $x \in X$, the intersection $B(x,d(x,A))\cap A$ is externally hyperconvex in $X$.
\end{enumerate}
\end{Prop}

\begin{proof}
	Let us denote the second copy of $X$ by $X'$ and for any $y \in X$, let $y'$ denote its corresponding copy in $X'$. Pick $x_0 \in X$ and $r_0 \ge 0$ such that $d(x_0,A) \leq r_0$ and let $\{(x_i,r_i)\}_{i \in I} \subset A \times [0,\infty)$ be such that $d(x_i,x_j) \leq r_i+r_j$ for $i,j \in I \cup \{0\}$. It follows that $d(x_0,x_0') \leq 2r_0$ and since $X \sqcup_A X$ is hyperconvex we have 
\[
	B := \bigcap_{i\in I} B(x_i,r_i) \cap B(x_0,r_0) \cap B(x_0',r_0)\neq \emptyset.
\]
	By symmetry there are $y,y' \in B$ with $y\in X$ and $y' \in X'$. Then, since intersections of balls are hyperconvex there is some geodesic $[y,y'] \subset B$,
	which must intersect $A$. Therefore we get 
\[
	\bigcap_{i\in I} B(x_i,r_i) \cap B(x_0,r_0) \cap A \neq \emptyset
\]
	and hence $A$ is weakly externally hyperconvex.
	
	For \textit{(ii)} observe that
\[
	B(x,d(x,A))\cap A=B(x,d(x,A))\cap B(x',d(x,A)) \in \mathcal{A}(X \sqcup_A X).
\] 
\end{proof}

Condition $(i)$ is not enough as the following example shows.

\begin{Expl}
Let $X_1$ and $X_2$ be two copies of $l_{\infty}^3$. Consider the gluing $X := X_1 \sqcup_V X_2$ where 
\[
V := \{x \in l_{\infty}^3 : x_1 = x_2 \text{ and } x_3 = 0\}.
\]
and where the gluing maps are given by the inclusion maps for $V$. To see that $X$ is not hyperconvex, consider $p_1:=(0,0,1)$ in $X_1$ as well as $p_2:=(2,0,0)$ and $p_2':=(0,-2,0)$ both in $X_2$. Note that $B(p_1,1) \cap X_2 = \{(t,t,0):t \in [-1,1]\}$ and hence $B(p_1,1) \cap B(p_2,1)=\{(1,1,0)\} \in V$ as well as $B(p_1,1) \cap B(p_2',1)=\{(-1,-1,0)\} \in V$. Moreover, $B(p_2,1) \cap B(p_2',1) = \{1\} \times \{-1\} \times [-1,1] \subset X_2$. Hence, 
\[
B(p_1,1) \cap B(p_2,1) \cap B(p_2',1) = \emptyset.
\] 
\end{Expl}

But as a consequence of Propositions~\ref{prop:gluing along WEH} and \ref{prop:gluing copies} we get the necessary and sufficient condition stated in Theorem~\ref{thm:gluing of copies}.

\begin{Expl}
	Let $A$ be a strongly convex subset of the hyperconvex space $X$ and $r \geq 0$. Then $B(A,r)$ is weakly externally hyperconvex and for any $x \in X$ the set $B(x,d(x,B(A,r)))\cap B(A,r)$ is externally hyperconvex in $X$ by Lemma~\ref{lem:nbhd of WEH}. Hence gluing along $B(A,r)$ preserves hyperconvexity.
\end{Expl}


\section{The Case of $l_{\infty}(I)$} 


A \emph{cuboid} in $l_\infty^n$ is the product $\prod_{i=1}^n I_i$ of closed (but not necessarily bounded) non-empty intervals $I_i \subset \mathbb{R}$.

\begin{Prop}\label{Prop:Cuboids}
	A subset $A$ of $l_\infty^n$ is externally hyperconvex if and only if it is a cuboid.
\end{Prop}

\begin{proof}
	On the one hand, cuboids are externally hyperconvex by Lemma~\ref{Lem:Products}. On the other hand, if $A$ is externally hyperconvex it is closed. Hence it is enough to show that for any points $x,y \in A$ and $z \in l_\infty^n$ with $z_i \in I(x_i,y_i)$ for each $i \in \{1, \ldots, n \}$, it follows that $z \in A$. Without loss of generality we may assume that $x_i \leq z_i \leq y_i$. Let 
	\[
	r := \max_{i \in \{1, \ldots, n \} } \{z_i-x_i, y_i-z_i\}.
	\]
	For each $i \in \{1, \ldots, n \}$, define 
	\begin{align*}
		p^i &= (z_1, \ldots z_{i-1}, x_i-r, z_{i+1}, \ldots, z_n ) \ \text{ and } \ r_i = z_i-x_i + r, \\
		q^i &= (z_1, \ldots z_{i-1}, y_i+r, z_{i+1}, \ldots, z_n ) \ \text{ and } \  s_i = y_i-z_i + r.
	\end{align*}
	Then, we have
	$$\bigcap_{i=1}^n B(p^i,r_i) \cap \bigcap_{i=1}^n B(q^i,s_i) = \{z\}.$$
	Moreover $d(p^i,A) \leq d(p^i,x) \leq r_i$ as well as $d(q^i,A) \leq d(q^i,y) \leq s_i$ and therefore since $A$ is externally hyperconvex
	$$\emptyset \neq \bigcap_{i=1}^n B(p^i,r_i) \cap \bigcap_{i=1}^n B(q^i,s_i) \cap A \subset \{z\}.$$
	It follows that $z \in A$ and this concludes the proof.
\end{proof}

The proof of Theorem~\ref{Thm:CellsInjectiveHullWEH} is a direct adaptation of the proof of \cite[Proposition~2.4]{Lan}. Note moreover that the proximinality assumption is no restriction since any weakly externally hyperconvex subset of a hyperconvex metric space is proximinal, cf. \cite{EspK}. Furthermore, if the index set $I$ is countable, $Q$ given by a system of inequalities as in Theorem~\ref{Thm:CellsInjectiveHullWEH} is always proximinal.

\begin{Lem}
Suppose that $Q$ is a non-empty subset of $l_\infty(I)$, with $I$ countable, given by an arbitrary system of inequalities of the form $\sig x_i \le C$ or $\sig x_i + \tau x_j \le C$ with $\sig,\tau \in \{ \pm 1 \} $ and $C \in \R$. Then, $Q$ is proximinal.
\end{Lem}
\begin{proof}
Recall that $l_1(I)^*$ is isomorphic to $l_\infty(I)$. Note that the maps from $l_{\infty}(I)$ to $\R$ given by $\varphi_{\sigma e_i}: f \mapsto \sigma f_i$ and $\varphi_{\sigma e_i + \tau e_j} : f \mapsto \sigma f_i + \tau f_j$ are continuous in the weak* topology on $l_{\infty}(I)$ since $\sigma e_i$ and $\sigma e_i + \tau e_j$ are both in $l_1(I)$. Since by assumption 
\[
Q = \bigcap_{(i,\sigma)} \varphi_{\sigma e_i}^{-1}((-\infty,C_{(i,\sigma)}]) \cap \bigcap_{(i,j,\sigma,\tau)} \varphi_{\sigma e_i + \tau e_j}^{-1}((-\infty,C_{(i,j,\sigma,\tau)}])
\]
we deduce that $Q$ is weak* closed and setting $\Delta := d(x,Q)$, it follows from the theorem of Banach-Alaoglu that the set $A:= B(x,\Delta+1) \cap Q$ is weak* compact. Now, the sets 
\[
(B(x,\Delta +\tfrac{1}{n}) \cap Q)_{n \in \mathbb{N}}
\]
form a decreasing sequence of weak*-closed non-empty subsets of $A$. By the closed set criterion for compact sets, it follows that
\[
B(x,\Delta) \cap Q = \bigcap_{n \in \mathbb{N}} \bigl(B(x,\Delta +\tfrac{1}{n}) \cap Q \bigr) \neq \emptyset.
\]
This shows that $Q$ is proximinal.
\end{proof}

We now turn to:

\begin{proof}[Proof of Theorem~\ref{Thm:CellsInjectiveHullWEH}]
For $i \in I$, denote by $R_i$ the reflection of $l_\infty(I)$ that 
interchanges $x_i$ with $-x_i$.

Let $A \cup \{b \}$ be a metric space with $\es \ne A$ and let $y \in l_\infty(I) \setminus Q$. For any $f \in \Lip_1(A,Q \cup \{y\})$ satisfying that $d_{\infty}(y,Q) \leq d(a,b)$ if $f(a)=y$, we show that there is an extension $\ol f \in \Lip_1(A \cup \{b \},Q \cup \{y\})$ such that $\ol f( b) \in Q$. It is easy to see that this implies that $Q \in \mathcal{W}(l_\infty(I))$, it is similar to showing that injectivity implies hyperconvexity, cf. \cite{Lan}.

Let $A' := A \setminus f^{-1}(\{y\})$. By proximinality of $Q$, we have $Q \cap B(y,d_\infty(y,Q)) \neq \es$. We can thus assume that 
\begin{equation} \label{eq:l0}
0 \in Q \cap B(y,d_{\infty}(y,Q)) \subset Q \cap \bigcap_{a' \in f^{-1}(\{y\})}B(f(a'),d(a',b)),
\end{equation}
so that all constants on the right sides of the inequalities describing $Q$ are non-negative. First, for a real valued function $f \in \Lip_1(A,\R)$, we combine the smallest and largest $1$-Lipschitz extensions and define $\ol f \colon A \cup \{b \} \to \R$ by
\[
\ol f(b) := \sup \Bigl\{ 0,\,\sup_{a \in A}(f(a) - d(a,b)) \Bigr\} + 
\inf \Bigl\{ 0,\,\inf_{a' \in A}(f(a') + d(a',b)) \Bigr\}.
\]
Note that at most one of the two summands is nonzero since 
$f(a) - d(a,b) \le f(a') + d(a,a') - d(a,b) \le f(a') + d(a',b)$.
It is not difficult to check that $\ol f$ is a $1$-Lipschitz extension 
of $f$ and that $\ol{R \circ f} = R \circ \ol f$ for the reflection 
$R \colon x \mapsto -x$ of $\R$. Moreover, by~\eqref{eq:l0}, it follows that 
\[
\sup_{a \in f^{-1}(\{y\})}(f(a) - d(a,b))  \le 0
\]
and
\[
\inf_{a \in f^{-1}(\{y\})}(f(a) + d(a,b))  \ge 0.
\]
It follows that $\ol f$ can be defined by taking suprema and infima on $A'$ instead of $A$ without changing $\ol f(b)$. We define the extension operator 
\[
\phi \colon \Lip_1(A,l_\infty(I)) \to \Lip_1(A \cup \{b \},l_\infty(I))
\]
such that $\phi(f)$ satisfies $\phi(f)_i = \ol{f_i}$ for every~$i$.
Clearly $\phi(f) \in \Lip_1(A \cup \{b \},l_\infty(I))$ and
\begin{equation} \label{eq:l1}
\phi(R_i \circ f) = R_i \circ \phi(f)
\end{equation}
for every $i$. 
To see that $\phi(f)(b) \in Q$, it thus suffices to show that the components of $\phi(f)$ satisfy
\begin{equation} \label{eq:l2}
\phi(f)_i + \phi(f)_j \le C
\end{equation}
whenever $f_i + f_j \le C$ for some pair of possibly equal indices $i,j$
and some constant $C \ge 0$.

Suppose that $f_i(a) + f_j(a) \le C$ for some indices $i,j$, some constant $C \ge 0$ and every $a \in A' := A \setminus f^{-1}(\{y\})$. Assume that  
$\phi(f)_i(b) \ge \phi(f)_j(b)$. If $\phi(f)_j(b) > 0$, then
\begin{align*}
\phi(f)_i(b) + \phi(f)_j(b)
&= \sup_{a,a' \in A'} (f_i(a) + f_j(a') - d(a,b) - d(a',b)) \\
&\le \sup_{a,a' \in A'} (f_i(a) + f_j(a') - d(a,a')) \\
&\le \sup_{a \in A'} (f_i(a) + f_j(a)) \le C.
\end{align*}
If $\phi(f)_i(b) > 0 \ge \phi(f)_j(b)$, then 
\begin{align*}
\phi(f)_i(b) + \phi(f)_j(b)
&\le \sup_{a \in A'} (f_i(a) - d(a,b)) + \inf_{a' \in A'} (f_j(a') + d(a',b)) \\
&\le \sup_{a \in A'} (f_i(a) + f_j(a)) \le C.
\end{align*}
Finally, if $\phi(f)_i(b) \le 0$, then 
$\phi(f)_i(b) + \phi(f)_j(b) \le 0 \le C$.
\end{proof}

\begin{Rem}\label{RemarkInjectiveHull}
By Theorem~\ref{Thm:CellsInjectiveHullWEH}, and borrowing the notation and terminology from \cite{Lan}, for any metric space $X$ and for any admissible set $A \in \cA(X)$, note  that the set 
\[
P(A) := \Delta(X) \cap H(A)
\]
is isometric to a weakly externally hyperconvex subset of $l_{\infty}(X)$ if it is proximinal. Indeed, for any point $x_0 \in X$, let us denote by $\tau \colon \R^X \to \R^X$ the translation $f \mapsto f - d_{x_0}$ where $d_{x_0}: x \mapsto d(x_0,x) $. One has
\[
\tau(P(A)) \subset \tau(\E X) \subset l_{\infty}(X),
\]
see \cite[Section~3~and~4]{Lan}. Recall that 
\begin{align*}
\Delta(X) &:= \{ f \in \R^X : f(x) + f(y) \ge d(x,y) \text{ for all } x,y \in X \} \text{ and } \\
H(A) &:= \{ g \in \R^X : g(x) + g(y) = d(x,y) \text{ for all } \{x,y\} \in A \}.
\end{align*}
By Theorem~\ref{Thm:CellsInjectiveHullWEH}, if $P(A)$ is proximinal, it follows $\tau(P(A))$ is in $\mathcal{W}(l_{\infty}(X))$, as claimed. In particular, it follows that $\tau(P(A))$ is weakly externally hyperconvex in $\tau(\E X)$ and since $\tau$ is an isometry, $P(A) \in \mathcal{W}(\E X)$. Finally, remember that under suitable assumptions on $X$, the family $\{P(A)\}_{A \in \cA(X)}$ endows $\E X$ with a canonical polyhedral structure where for any $A \in \cA(X)$, $P(A)$ is a cell of $\E X$ isometric to a convex polytope in $l^n_{\infty}$ hence it is compact and thus in particular proximinal and Theorem~\ref{Thm:CellsInjectiveHullWEH} applies.
\end{Rem}

In the following, we use the notation $I_n :=\{1,\dots,n\}$.

\begin{Thm}\cite[Theorem~2.1]{Pav}\label{Thm:HyperconvexSubspaces}
Let $\emptyset \neq V \subset l_{\infty}^n$ be a linear subspace and let $k:= \mathrm{dim}(V)$. Then, the following are equivalent:
\begin{enumerate}[$(i)$]
\item $V$ is hyperconvex.
\item There is a subset $J \subset I_n$ with $|J|=k$ such that for any $i \in I_n \setminus J$ there exist real numbers $\{c_{i,j}\}_{j \in J}$ such that $\sum_{j \in J} |c_{i,j}| \le 1$ and
\[
V = \Biggl \{ (x_1,\dots,x_n) \in l_{\infty}^n : \text{ for all } i \in I_n \setminus J \ , \text{ one has } x_i = \sum_{j \in J} c_{i,j}x_j \Biggr \}.
\]
\end{enumerate}
\end{Thm}

\begin{Prop}\label{Prop:StronglyConvex}
Let $n \ge 1$ and let $V$ denote any linear subspace of $l_{\infty}^n$. Then, $V$ is strongly convex if and only if one of the following three cases occurs:
\begin{enumerate}[(i)]
\item $V=\{0\}$,
\item $V=\R e$ where $e$ denotes a vertex of the hypercube $[-1,1]^n$ or
\item $V=l_{\infty}^n$.
\end{enumerate}
\end{Prop}

\begin{proof}
Let $k=\dim V$. Since strongly convex subspaces are hyperconvex, by Theorem~ \ref{Thm:HyperconvexSubspaces}, we may assume that $V$ is of the form
\[
	V = \left \{ (x_1, \ldots, x_n) \in l_\infty^n : x_l = \sum_{i=1}^k c_{l,i} x_i \text{ for } l=k+1, \ldots n \right \}
\]
with $\sum_{i=1}^k |c_{l,i}| \leq 1$. Fix some $1 \leq i \leq k$ and consider 
\[
	x=(0, \ldots 0,1,0, \ldots 0, c_{k+1,i}, \ldots, c_{n,i}) \in V.
\]
We have $c_{l,i} \neq 0$ since otherwise
\[
	y=(0, \ldots 0,\tfrac{1}{2},0, \ldots 0, \tfrac{c_{k+1,i}}{2}, \ldots,\tfrac{c_{l-1,i}}{2},\tfrac{1}{2},\tfrac{c_{l+1,i}}{2}, \ldots, \tfrac{c_{n,i}}{2}) \in I(0,x) \setminus V.
\]
Now take $l$ such that $|c_{l,i}|$ is minimal. Then
\[
	y=(0, \ldots 0,|c_{l,i}|,0, \ldots 0, \tfrac{|c_{l,i}|}{|c_{k+1,i}|} c_{k+1,i}, \ldots, \tfrac{|c_{l,i}|}{|c_{n,i}|} c_{n,i}) \in I(0,x) \subset V
\]
and therefore $c_{l,i} = c_{l,i} \cdot |c_{l,i}|$, i.e. $|c_{l,i}|=1$. Hence one has $c_{l,i} \in \{\pm 1\}$ for all $l=k+1, \ldots, n$ and $i=1, \ldots, k$. Since $\sum_{i=1}^k |c_{l,i}| \le 1$, it thus follows that either $k=0$, which corresponds to $(i)$, $k=1$, which corresponds to $(ii)$ or $k =n$, which corresponds to $(iii)$.
\end{proof}

In the following, $\mathrm{relint}(S)$ denotes the relative interior of $S$ and $F_i$ denotes the facet $[-1,1]^n \cap \{x \in l_{\infty}^n : x_i=1\}$ of the unit ball $[-1,1]^n$ in $l_{\infty}^n$. Half-spaces are denoted by $H_v := \{ x \in \R^n : x \cdot v \ge 0 \}$ (with $ \cdot $ denoting the standard scalar product) and $\partial H_v$ being the boundary of $H_v$.

\begin{Lem}\label{Lem:IntersectingFacetsUnitCube}
Let $V$ denote any hyperconvex linear subspace of $l_{\infty}^n$ such that $V$ is not contained in any hyperplane of the form $\partial H_{\sigma e_i + \tau e_j}$. Then, the following hold:
\begin{enumerate}[(i)]
\item There is $(i,\sigma) \in I_n \times \{\pm 1\}$ such that $V \cap \mathrm{relint}(\sigma F_i) \neq \emptyset$. \label{it:IntersectingRelativeInterior}
\item If $V \neq l_{\infty}^n$, there is $i \in I_n$ and $\nu \in \R^n \setminus \{0\}$ such that $V \subset \partial H_{\nu}$ and
\[
\partial H_{\nu} \cap (F_i \cup (-F_i)) = \emptyset.
\]\label{it:NoIntersectingFacet}
\end{enumerate}
\end{Lem}

\begin{proof}
By Theorem~ \ref{Thm:HyperconvexSubspaces}, $V$ can without loss of generality be written as
\[
V:= \left \{(x_1,\dots,x_k,\sum_{j=1}^k c_{k+1,j}x_j,\dots,\sum_{j=1}^k c_{n,j}x_j) : x_1,\dots,x_k \in \R \right \},
\]
where for every $m \in \{ k+1,\dots,n\}$, one has $\sum_{j=1}^k |c_{m,j}| \le 1$. It follows that $ V \subset \partial H_{\nu^m} $, where 
\[
\nu^m := (c_{m,1},\dots,c_{m,k},0,\dots,0,-1,0,\dots,0)
\]
(i.e. $-1$ in the $m$-th entry) and $\left\| \nu^m \right\|_1 = 1 + \sum_{j=1}^k |c_{m,j}| \le 2 = 2 \left\| \nu^m \right\|_{\infty}$. Since we assume that $V \not \subset  \partial H_{\sigma e_i + \tau e_j}$, it follows that for every $m \in \{ k+1,\dots,n\}$ and every $i \in \{1,\dots,k\}$ we have $|c_{m,i}| < 1$.
In particular, \eqref{it:IntersectingRelativeInterior} follows from the fact that
\[
(1,0,\dots,0,c_{k+1,1},\dots,c_{n,1}) \in V \cap \mathrm{relint}(F_1).
\]
Now, \eqref{it:NoIntersectingFacet} follows from the fact that $ V \subset \partial H_{\nu} $ where 
\[
\nu := (c_{k+1,1},\dots,c_{k+1,k},-1,0,\dots,0) \in \mathrm{relint}(-F_{k+1})
\]
and
\[
\partial H_{\nu} \cap (F_{k+1} \cup (-F_{k+1})) = \emptyset.
\]
\end{proof}

We now proceed to the characterization of weakly hyperconvex polyhedra in $l^n_{\infty}$.

\begin{Thm}\label{Thm:FiniteDimensionalWEHLinearSubspaces}
Let $n \ge 1$ and let $V$ denote any linear subspace of $l_{\infty}^n$. Then, $V$ is weakly externally hyperconvex in $l_{\infty}^n$ if and only if $V$ can be written as the intersection of hyperplanes of the form $\partial H_{\sigma e_i}$ or $\partial H_{\sigma e_i + \tau e_j}$ where $\sigma, \tau \in \{\pm 1\}$ and $i,j \in \{1, \cdots , n \} $.
Equivalently, $V$ is weakly externally hyperconvex in $l_{\infty}^n$ if and only if 
\[
V = \{0\}^k \times l^m_{\infty} \times S_{p_1} \times \cdots \times S_{p_q},
\]
where $S_{p_j}$ denotes a one-dimensional strongly convex linear subspace of $l^{p_j}_{\infty}$ for each $j \in \{1,\dots,q\}$. Note that then $n = k + m + \sum_j p_j$ and $\mathrm{dim}(V) = m + q$.
\end{Thm}

\begin{proof}

First, it is easy to see that both representations for $V$ given in the statement of Theorem~\ref{Thm:FiniteDimensionalWEHLinearSubspaces} are equivalent. To see this, remark that if $V$ has the desired product representation, then using Proposition~\ref{Prop:StronglyConvex}, it is easy to express $V$ as an intersection of the desired form. Conversely, if $V$ can be written as such an intersection, then $V$ is given by a system of linear equalities which one can decompose into a set of minimal linear subsystems. Applying Proposition~\ref{Prop:StronglyConvex}, one obtains the desired product representation for $V$.

Note that the hyperplanes of $l^n_{\infty}$ of the form $\partial H_{\sigma e_i}$ or $\partial H_{\sigma e_i + \tau e_j}$ are exactly the ones to which Theorem~\ref{Thm:CellsInjectiveHullWEH} applies. Hence, whenever $V$ can be written as an intersection of such hyperplanes, it follows that $V$ is weakly externally hyperconvex. This proves one implication in Theorem~\ref{Thm:FiniteDimensionalWEHLinearSubspaces}.

We now assume that $V \in \mathcal{W}(l_{\infty}^n)$ and we show that $V$ can be written as \[
V = \bigcap_{(i,\sigma)} \partial H_{\sigma e_i} \cap  \bigcap_{(i,\sigma),(j,\tau)} \partial H_{\sigma e_i + \tau e_j}.
\]
If $V \subset W$, then $V \in \mathcal{W}(W)$. Now assume that $W = \partial H_{e_i}$ or $W= \partial H_{\sigma e_i + \tau e_j}$. Then there is a canonical bijective linear isometry 
\[
	\pi \colon W \to l_{\infty}^{n-1}, (x_1, \ldots, x_n) \mapsto (x_1, \ldots, x_{i-1},x_{i+1}, \ldots x_n)
\]
with $\pi(V) \in \mathcal{W}(l_{\infty}^{n-1})$ and we can iterate this process.

We can therefore without loss of generality assume that $V$ is not contained in any hyperplane of the form $\partial H_{\sigma e_i}$ or $\partial H_{\sigma e_i + \tau e_j}$. It follows by \eqref{it:IntersectingRelativeInterior} in Lemma~\ref{Lem:IntersectingFacetsUnitCube} and hyperconvexity of $V$ that there is $(i,\sigma) \in I_n \times \{\pm 1\}$ such that $V \cap \mathrm{relint}(\sigma F_i) \neq \emptyset$.

By \eqref{it:NoIntersectingFacet} in Lemma~\ref{Lem:IntersectingFacetsUnitCube}, there is a facet $\tau F_j$ of $[-1,1]^n$ such that $\tau F_j \cap V = \emptyset$. It follows that the facet $ \tau F_j':=\tau F_j \cap \sigma F_i$ of $\sigma F_i$ satisfies $\tau F_j' \cap V = \emptyset$.

But $\tau F_j'$ can be written as $\tau F_j'= \cap_{i \in \{1,2,3\}} B_i$ where we pick $v^{(0)} \in V \cap \mathrm{relint}(\sigma F_i)$ so that $[0,\infty)v^{(0)} \cap \partial H_{\tau e_j} \neq \emptyset$ and with
\begin{align*}
B_1 &= B(0,1), \\
B_2 &= B(R v^{(0)}, \| R v^{(0)} \|_{\infty} ) \text{ and } \\
B_3 &= B(\sigma e_i + \tau e_j + \widetilde{R} \tau e_j , \| \sigma e_i + \tau e_j + \widetilde{R} \tau e_j - ( \sigma e_i + \tau e_j ) \|_{\infty} ).
\end{align*}
Indeed, for $\widetilde{R}$ big enough, one has $B_3 \cap V \neq \emptyset$ since $V \not \subset \partial H_{e_j}$ and by the choice of $v^{(0)}$. For $\widetilde{R}$ chosen even bigger (if needed) and for $R$ big enough, one has $\sigma F_i = B_1 \cap B_2$ and $\tau F_j' = B_1 \cap B_2 \cap B_3$. It follows that the balls pairwise intersect and $V \cap B_1 \cap B_2 \cap B_3 = V \cap \tau F_j' = \emptyset$. This implies that $V \not \subset \mathcal{W}(l_{\infty}^n)$ which is a contradiction.
\end{proof}

Let $V$ be a finite dimensional real vector space and let $Q$ be any convex non-empty subset of $V$. The \textit{tangent cone} $\mathrm{T}_pQ$ to $Q$ at $p \in Q$ is defined to be the set $\overline{\cup_{n \in \mathbb{N}} ( p + n ( Q-p) ) }$ where $Q-p := \{ q-p : q \in Q\}$ and $nQ:=\{ nq : q \in Q\}$.

We can now proceed to the following:

\begin{Thm}\label{Thm:FiniteDimensionalWEHConvexPolyhedra}
Suppose that $Q$ is a convex polyhedron in $l_{\infty}^n$ with non-empty interior.
Then, the following are equivalent:
\begin{enumerate}[(i)]
\item $Q$ is given by a finite system of inequalities of the form $\sig x_i \le C$ or $\sig x_i + \tau x_j \le C$ with $|\sig|,|\tau| = 1$ and $C \in \R$, \label{it:ConvexPolyhedraItem1}
\item $Q \in \mathcal{W}(l_{\infty}^n)$.\label{it:ConvexPolyhedraItem2}
\end{enumerate}
\end{Thm}

\begin{proof}
The fact that \eqref{it:ConvexPolyhedraItem1} implies \eqref{it:ConvexPolyhedraItem2} is a direct consequence of Theorem~\ref{Thm:CellsInjectiveHullWEH}. On the other hand, the fact that \eqref{it:ConvexPolyhedraItem2} implies \eqref{it:ConvexPolyhedraItem1} follows by considering for each point $p$ in the relative interior of a facet of $Q$, the tangent cone $\mathrm{T}_pQ$. By the choice of $p$, $\mathrm{T}_pQ = H$ where $H$ is a closed half-space and from Proposition~\ref{Prop:IncreasingSequenceWEH} it follows that $H \in \mathcal{W}(l_{\infty}^n)$. Furthermore, $\partial H = H \cap (-H) \in \mathcal{W}(l_{\infty}^n)$ by Proposition~\ref{prop:IntersectionWEH}. Now, Theorem~\ref{Thm:FiniteDimensionalWEHLinearSubspaces} implies that $\partial H$ is given by an equality of the form $\sig x_i = C$ or $\sig x_i + \tau x_j = C$ with $\sigma,\tau \in \{ \pm 1 \}$ and $C \in \R$. Hence, $H$ is given by an inequality of the desired form. Since $Q$ can be written as the intersection of all such tangent cones $\mathrm{T}_pQ$, the result follows. 
\end{proof}

\begin{Lem}\label{Lem:lem1}
	Let $X_1=l_\infty^n$ and $X_2=l_\infty^m$. Moreover let $V$ be a linear subspace of both $X_1$ and $X_2$. Then for $x\in X_1$ and $r\geq 0$ we have
\[
	B(x,r) \cap X_2 = \bigcup_{y \in B^1(x,r) \cap V} B^2(y,r-d(x,y)).
\]
	Moreover, $B(x,r) \cap X_2$ is connected.
\end{Lem}

\begin{proof}
	Since $V$ is proper, for any $ x \in X_1$ and $ x'\in X_2$ there is some $y \in V$ with $d(x,x')=d(x,y)+d(y,x')$. Since $B^1(x,r) \cap V$ is connected, this is also true for $B(x,r) \cap X_2$.
\end{proof}

If $V$ is any finite dimensional real vector space and $Q$ is any convex non-empty subset of $V$, we say that $e \in Q$ is an \textit{extreme point} of $Q$, and we write $e \in \mathrm{ext}(Q)$, if for any $c,c' \in Q$ and $t \in [0,1]$ such that $e = (1-t) c + t c'$, it follows that $t \in \{0,1\}$. Finally, by a \textit{convex polytope} in $V$, we mean a bounded intersection of finitely many closed half-spaces or equivalently the convex hull of finitely many points. A compact convex set in $V$ is a convex polytope if and only if it has finitely many extreme points (cf. \cite[Section~3.1]{Gru}).

\begin{Lem}\label{Lem:lem2}
	Let $X_1=l_\infty^n$ and $X_2=l_\infty^m$. Moreover let $V$ be a linear subspace of both $X_1$ and $X_2$ such that $V \neq X_1,X_2$ and
\[
X = X_1 \sqcup_V X_2
\]
is hyperconvex. Then for $x\in X_1$ and $r > d(x,V)$ the set $B(x,r) \cap X_2$ is a cuboid.
\end{Lem}

\begin{proof}
We first show that $Q := B(x,r) \cap X_2$ is convex. Let $p,q \in Q$, there are then $\bar{p},\bar{q} \in V$ such that 
\begin{equation}\label{eq:EquationIntervals}
d(x,\bar{p}) + d(\bar{p},p)= d(x,p) \ \text{  as well as  } \ d(x,\bar{q}) + d(\bar{q},q)= d(x,q).
\end{equation}
Consider now any point $z:= p + t(q-p) \in p + [0,1](q-p) \subset X_2$ and let $\bar{z}:= \bar{p} + t(\bar{q}-\bar{p}) \in V$. On the one hand
\[
d(z, \bar{z}) \le (1-t) d( p, \bar{p}) + t d( q, \bar{q})
\]
and on the other hand
\[
d(x,\bar{z}) \le (1-t) d( x, \bar{p}) + t d( x, \bar{q}).
\]
Summing the above two inequalities, applying the triangle inequality and using \eqref{eq:EquationIntervals}, we obtain
\[
d(x,z) \le  (1-t) d(x,p) + t d(x,q) \le r.
\]
It follows that $z \in Q$ and shows that $Q$ is convex. 

Since $X$ is hyperconvex, for every $y \in X_2$ with $d(x,y) \leq r + d(y,V)$, the set $B(y,d(y,V)) \cap B(x,r) \cap X_2 = B(y,d(y,V)) \cap B(x,r) \in \mathcal{E}(X_2)$ and is therefore a cuboid. Hence $B(x,r) \cap X_2 \setminus V$ is locally a cuboid.

To see that $Q$ is a convex polytope, we show that $Q$ has only finitely many extreme points. Observe that $Q$ has only finitely many extreme points inside $V$ since $Q \cap V = B^1(x,r) \cap V$ and the right-hand side is a convex polytope contained in $V$. Now, note that since $Q$ is locally a cuboid outside of $V$, $Q$ has only finitely many different tangent cones at points outside of $V$, and this up to translation (of such cones). It is easy to see that $Q$ cannot have both a cone and a nontrivial translate of this cone as tangent cones. It follows that $Q$ can only have finitely many different tangent cones at points outside $V$. In particular, $Q$ has only finitely many extreme points and it is thus a convex polytope. Furthermore, $V$ meets the interior of $Q$ and $V$ has codimension at least one in $X_2$, hence $V$ cannot contain any facet of $Q$. It follows that $Q$ can be written as the intersection of tangent cones to $Q$ at points in $Q \setminus V$. Once again, since $Q$ is locally a cuboid outside of $V$, it follows that $Q$ is a cuboid.
\end{proof}

\begin{Lem}\label{Lem:lem3}
	Let $X_1=l_\infty^n$ and $X_2=l_\infty^m$. Moreover let $V$ be a linear subspace of both $X_1$ and $X_2$ such that $V \neq X_1,X_2$ and
\[
X = X_1 \sqcup_V X_2
\]
is hyperconvex. Then $V$ must be weakly externally hyperconvex in $X$.
\end{Lem}

\begin{proof}
	Note that it is enough to show that $V$ is weakly externally hyperconvex in $X_1$.
	First we introduce coordinates: 
\[
	V= \left \{ x \in l_\infty^m : \text{ for every } l \in \{ k+1, \ldots, n \}, \text{ one has } x_l = \sum_{i=1}^k c_{l,i} x_i  \right \}.
\]
	As $V$ must be hyperconvex by Lemma~\ref{lem:list properties gluing}, we may assume that $\sum_{i=1}^k |c_{l,i}| \leq 1$ for all $l$ by Theorem~ \ref{Thm:HyperconvexSubspaces}. For $x \in X_1$ and $r > d(x,V)$ and by Lemma~\ref{Lem:lem2}, we can write
\[
	A:=B(x,r) \cap X_2 = \prod_{i=1}^m [a_i^{-1},a_i^{1}].
\]
	Let $\nu$ be a vertex of $[-1,1]^m$, i.e. $\nu \in \{-1,1\}^m$, and $a^\nu = (a_1^{\nu_1},\dots,a_m^{\nu_m} )$ the corresponding vertex of $A$.
	By Lemma~\ref{Lem:lem1}, for every $a^\nu$ there must be some $\bar{a}^\nu \in V$ such that $d(a^\nu,\bar{a}^\nu) = r - d(x,\bar{a}^\nu)$. Since $a^\nu$ is a vertex of the cuboid $A$, $\bar{a}^\nu$ must lie on the diagonal $a^\nu + \mathbb{R}\nu$ or more precisely 
\begin{align}\label{eq:corner}
	\bar{a}^\nu = a^\nu - t_\nu \nu
\end{align}	
	for $t_\nu = d(a^\nu, \bar{a}^\nu) \geq 0$. Hence for every $l \in \{k+1, \ldots, m \}$, one has
\begin{align}\label{eq:in V}
	\bar{a}_l^\nu = \sum_{i=1}^k c_{l,i} \bar{a}_i^{\nu_i} = \sum_{i=1}^k c_{l,i} a_i^{\nu_i} - t_\nu \sum_{i=1}^k c_{l,i} \nu_i.
\end{align}
Fix now some $l \in \{k+1, \ldots, m \}$. Then for $\nu \in \{-1,1\}^m$ consider $\nu^-, \nu^+ \in \{-1,1\}^m$ with $\nu^+_l =1$, $\nu^-_l = -1$ and $\nu^+_i = \nu^-_i =\nu_i$ for $i \neq l$. Clearly $\nu$ coincides with either $\nu^+$ or $\nu^-$. From \eqref{eq:corner} we get
\begin{align*}
	\bar{a}_l^{\nu^+} - \bar{a}_l^{\nu^-} 
	= (a_l^{\nu^+_l} - t_{\nu^+}\nu^+_l) - (a_l^{\nu^-_l} - t_{\nu^-}\nu^-_l) 
	= a_l^1 - a_l^{-1} - t_{\nu^+}- t_{\nu^-}.
\end{align*}
Observe that for $\nu^+, \nu^-$ on the right-hand side of (\ref{eq:in V}) all parameters except $t_{\nu^+}, t_{\nu^-}$ are equal. Hence setting $c=\sum_{i=1}^k c_{l,i} \nu_i \in [-1,1]$ we get
\[
\bar{a}_l^{\nu^+} - \bar{a}_l^{\nu^-} = c(t_{\nu^+}- t_{\nu^-}) \leq |t_{\nu^+}- t_{\nu^-}|.
\]
Let us assume that $t_{\nu^+} \geq t_{\nu^-}$ (the other case is analogous). We then get 
\begin{align*}
a_l^{+1}-a_l^{-1} - t_{\nu^+} - t_{\nu^-} &\leq t_{\nu^+}-t_{\nu^-}, 
\end{align*}
or equivalently $2t_{\nu^+} \geq a_l^{+1}-a_l^{-1}$. But this inequality must be an equality since $B^2(\bar{a}^{\nu^+},t_{\nu^+}) \subset A$. Moreover we have $a^{\nu^+}, a^{\nu^-} \in B^2(\bar{a}^{\nu^+},t_{\nu^+})$. We conclude that $\bar{a}^\nu$ can always be chosen such that 
\[
d(a^\nu,\bar{a}^\nu) = t_\nu = \frac{1}{2}\left(a_l^{+1}-a_l^{-1}\right).
\]
Especially we have $\bar{a}_l^\nu = \frac{1}{2}\left(a_l^{+1}+a_l^{-1}\right)$. Hence, all vertices of $A$, and by convexity all points $a \in A$ are contained in a ball $B^2(\bar{a},t)$ with $\bar{a} \in B^1(x,r)\cap V$ and maximal radius 
\[
t=r-d(x,\bar{a})=\frac{1}{2}\left(a_l^{+1}-a_l^{-1}\right).
\]
Note that for $y \in B^1(x,r)\cap V$ the radius $r-d(x,y)$ is maximal if and only if $d(x,y)=d(x,V)$ and therefore $\bar{a} \in B(x,d(x,V))\cap V$.

We now use Lemma~\ref{lem:characterization WEH} to conclude that $V$ is weakly externally hyperconvex in $X_1$. First we have
\[
	B(x,d(x,V)) \cap V = \bigcap_n B(x,d(x,V)+\tfrac{1}{n}) \cap X_2 \in \mathcal{E}(X_2)
\]
using Lemma~\ref{Lem:lem2} and therefore $B(x,d(x,V)) \cap V \in \mathcal{E}(V)$. Furthermore, for $y \in V$ assume that $r=d(x,y) > d(x,V)$. Then by the above there is some $a \in V \cap B(x,d(x,V))$ such that 
\[
	r = d(x,y) \leq d(x,a) + d(a,y) \leq d(x,V) + r - d(x,V) = r,
\]
i.e. $d(x,y) = d(x,a) + d(a,y)$.
\end{proof}

Finally,

\begin{proof}[Proof of Theorem~\ref{Thm:CharacterizationGluingLInfinity}]
	If $X$ is hyperconvex, by Lemma~\ref{Lem:lem3}, the subspace $V$ is weakly externally hyperconvex in $X_1$ and in $X_2$, we can thus apply Theorem~\ref{Thm:FiniteDimensionalWEHLinearSubspaces} and write 
\[
	V = l_\infty^{k_1} \times S_0 \times S_1 \times \ldots \times S_p \subset l_\infty^{k_1} \times l_\infty^{\mu_0} \times l_\infty^{\mu_1} \times \ldots \times l_\infty^{\mu_p} = X_1
\]
	where, for all $i$, the set $S_i \subsetneq l_\infty^{\mu_i}$ is strongly convex and, for $i \neq 0$, $S_i$ is one dimensional (for $i = 0$ we also might have $S_0=\{0\}$) and similarly
\[
	V = l_\infty^{k_2} \times T_0 \times T_1 \times \ldots \times T_q \subset l_\infty^{k_2} \times l_\infty^{\nu_0} \times l_\infty^{\nu_1} \times \ldots \times l_\infty^{\nu_q} = X_2.
\]
	
	Let $\{e_1, \ldots, e_n \}$ and $\{ f_1, \ldots, f_m\}$ be the standard basis for $X_1$ and $X_2$ respectively. First observe that if $\mathbb{R}e_i \subset V$ then there is some $f_j$ with $\mathbb{R}e_i=\mathbb{R}f_j \subset V$. Indeed, if $\mathbb{R}e_i \subset V$ for every $r \geq 0$ we have
\[
 [-r,r]e_i = \bigcap_{e \in \{-1,1\}^{i-1}\times \{0\} \times \{-1,1\}^{n-i} \subset X_1} B(r e,r) \in \mathcal{E}(X).
\]
Note that as $[-r,r]e_i \in \mathcal{E}(X_2)$ it is a cuboid in $X_2$ and therefore $[-r,r]e_i=[-r,r]f_j$.

Hence we have $k = k_1 = k_2$ and can split off the maximal component $l_\infty^k$ from $X_1$, $X_2$ and $V$, that is we can write $V := l_\infty^k \times V' \subset l_\infty^k \times X_i' = X_i$ with
\begin{align*}
	V'= S_0 \times S_1 \times \ldots \times S_p \subset l_\infty^{\mu_0} \times l_\infty^{\mu_1} \times \ldots \times l_\infty^{\mu_p} = X_1'
\end{align*}	
as well as
\begin{align*}
	V'=T_0 \times T_1 \times \ldots \times T_q \subset l_\infty^{\nu_0} \times l_\infty^{\nu_1} \times \ldots \times l_\infty^{\nu_q} = X_2'.
\end{align*}
		
Assume now by contradiction that there are at least two non-trivial factors in the first decomposition, i.e. $p \geq 1$. Since for every factor $l_\infty^{\mu_i}$ there is some $x_i \in l_\infty^{\mu_i}$ with $d(x_i,S_i)=1$ there is some $x=(x_0,x_1, \ldots, x_{p-1},0) \in X_1'$ with 
\[
	B(x,d(x,V')) \cap V' = B(x,1)\cap V' = \{0\}^p \times \left(S_p \cap B^{S_p}(0,1)\right),
\]	
where $B^{S_p}(0,1)$ denotes the unit ball inside $S_p$. But $\{0\} \times (S_p \cap B^{S_p}(0,1)) \subset X_2'$ is not externally hyperconvex in $X_2'$ since it is not a cuboid. It follows that $B(x,d(x,V')) \cap V' \notin \mathcal{E}(X_2')$ which is a contradiction to Proposition~\ref{prop:externally hyperconvex balls}. Therefore, $V$ must consist of only one strongly convex factor.

For the other direction note that $X = l_\infty^k \times (X_1'\sqcup_{V'} X_2')$. The second factor is hyperconvex by Theorem~\ref{thm:gluing along WEH} since $V'$ is strongly convex. Hence also the product is hyperconvex by Lemma~\ref{Lem:Products}.
\end{proof}

\textbf{Acknowledgments.} We would like to thank Prof. Dr. Urs Lang for useful remarks on this work. We gratefully acknowledge support from the Swiss National Science Foundation.



\begin{thebibliography}{99}

\bibitem{Aro}
N. Aronszajn and P. Panitchpakdi, 
Extension of uniformly continuous transformations and hyperconvex metric spaces.
Pacific J. Math. 6 (1956), 405–439. 

\bibitem{Bri}
M. Bridson and A. Haefliger,
Metric spaces of non-positive curvature.
Springer-Verlag, 1999.

\bibitem{Esp}
R. Esp\' inola,
On selections of the metric projection and best proximity pairs in hyperconvex spaces.
Ann. Univ. Mariae Curie-Skłodowska Sect. A 59 (2005), 9–17. 

\bibitem{EspK}
R. Esp\' inola and M. A. Khamsi,
Introduction to hyperconvex spaces.
Handbook of metric fixed point theory, 391–435, Kluwer Acad. Publ., Dordrecht, 2001. 

\bibitem{EKL}
R. Esp\' inola, W.A. Kirk and G. L\' opez,
Nonexpansive Retractions in Hyperconvex Spaces.
J. Math. Anal. Appl. 251 (2000), 557--570.

\bibitem{Gru}
B. Gr\"unbaum,
Convex Polytopes,
Springer, second edition, 2002.

\bibitem{Lan}
U. Lang,
Injective hulls of certain discrete metric spaces and groups.
J. Topol. Anal. 5 (2013), no. 3, 297--331.

\bibitem{LanPZ}
U. Lang, M. Pav\' on and R. Z\"ust,
Metric stability of trees and tight spans.
Arch. Math. (Basel) 101 (2013), 91--100.

\bibitem{Mie}
B. Miesch,
Gluing hyperconvex metric spaces. 
Anal. Geom. Metr. Spaces 3 (2015), 102--110.

\bibitem{Pav}
M. Pav\' on,
Injective Convex Polyhedra. 
arXiv:1410.7306.



\end{thebibliography}
\end{document}